\documentclass[reqno,10pt,centertags]{amsart}
\usepackage{amsmath,amsthm,amscd,amssymb,latexsym,esint,upref,stmaryrd,
enumerate,color,verbatim,yfonts}
\usepackage{hyperref} 
\newcommand*{\mailto}[1]{\href{mailto:#1}{\nolinkurl{#1}}}
\newcommand{\arxiv}[1]{\href{http://arxiv.org/abs/#1}{arXiv:#1}}

\newcommand{\R}{{\mathbb R}}

\newcommand{\C}{{\mathbb C}}

\newcommand{\bbC}{{\mathbb{C}}}

\newcommand{\bbN}{{\mathbb{N}}}

\newcommand{\bbR}{{\mathbb{R}}}

\newcommand{\bbZ}{{\mathbb{Z}}}

\newcommand{\bsA}{{\boldsymbol{A}}}
\newcommand{\bsB}{{\boldsymbol{B}}}

\newcommand{\bsD}{{\boldsymbol{D}}}

\newcommand{\bsH}{{\boldsymbol{H}}}
\newcommand{\bsI}{{\boldsymbol{I}}}

\newcommand{\bsT}{{\boldsymbol{T}}}

\newcommand{\cB}{{\mathcal B}}
\newcommand{\cC}{{\mathcal C}}

\newcommand{\cH}{{\mathcal H}}

\newcommand{\cJ}{{\mathcal J}}
\newcommand{\cK}{{\mathcal K}}

\newcommand{\beq}{\begin{equation}}
\newcommand{\enq}{\end{equation}}

\DeclareMathOperator{\esssup}{ess.sup}

\DeclareMathOperator{\dom}{dom}

\DeclareMathOperator{\tr}{tr}

\DeclareMathOperator*{\nlim}{n-lim}
\DeclareMathOperator*{\slim}{s-lim}
\DeclareMathOperator*{\wlim}{w-lim}

\DeclareMathOperator*{\sgn}{sgn}

\renewcommand{\Im}{\text{\rm Im}}
\renewcommand{\ln}{\text{\rm ln}}

\newcommand{\loc}{\operatorname{loc}}

\newcommand{\Lxi}{\xi_L}

\newcommand{\ind}{\operatorname{index}}
\newcommand{\no}{\notag}
\newcommand{\lb}{\label}
\newcommand{\f}{\frac}

\newcommand{\ol}{\overline}

\newcommand{\wti}{\widetilde}
\newcommand{\Oh}{O}

\newcommand{\hatt}{\widehat} 
\newcommand{\bi}{\bibitem}

\let\geq\geqslant
\let\leq\leqslant

\makeatletter
\def\theequation{\@arabic\c@equation}

\allowdisplaybreaks 
\numberwithin{equation}{section}

\newtheorem{theorem}{Theorem}[section]

\newtheorem{lemma}[theorem]{Lemma}
\newtheorem{corollary}[theorem]{Corollary}
\newtheorem{definition}[theorem]{Definition}
\newtheorem{hypothesis}[theorem]{Hypothesis}

\theoremstyle{remark}
\newtheorem{remark}[theorem]{Remark}

\begin{document}

\title[On the Index of a Non-Fredholm Model Operator]{On the Index of a Non-Fredholm Model Operator} 

\author[A.\ Carey]{Alan Carey}  
\address{Mathematical Sciences Institute, Australian National University, 
Kingsley St., Canberra, ACT 0200, Australia
and School of Mathematics and Applied Statistics, University of Wollongong, NSW, Australia,  2522}  
\email{\mailto{acarey@maths.anu.edu.au}}
\urladdr{\url{http://maths.anu.edu.au/~acarey/}}
  
\author[F.\ Gesztesy]{Fritz Gesztesy}  
\address{Department of Mathematics,
University of Missouri, Columbia, MO 65211, USA}
\email{\mailto{gesztesyf@missouri.edu}}
\urladdr{\url{https://www.math.missouri.edu/people/gesztesy}}

\author[G.\ Levitina]{Galina Levitina} 
\address{School of Mathematics and Statistics, UNSW, Kensington, NSW 2052,
Australia} 
\email{\mailto{g.levitina@student.unsw.edu.au}}

\author[F.\ Sukochev]{Fedor Sukochev}
\address{School of Mathematics and Statistics, UNSW, Kensington, NSW 2052,
Australia} 
\email{\mailto{f.sukochev@unsw.edu.au}}

\dedicatory{Dedicated to the memory of Leiba Rodman (1949--2015).}

\thanks{To appear in {\it Operators and Matrices}. {\bf }}

\date{\today}

\begin{abstract} 
Let $\{A(t)\}_{t \in \bbR}$ be a  path of self-adjoint Fredholm operators 
in a Hilbert space $\cH$, joining endpoints $A_\pm$ as $t \to \pm \infty$. Computing 
the index of the operator $\bsD_{\bsA}=\partial/\partial t + \bsA$ acting on 
$L^2(\bbR; \cH)$, where $\bsA$ denotes the multiplication operator
$(\bsA f)(t) = A(t)f(t)$ for $f\in L^2(\bbR; \cH)$, and its relation to spectral flow along this path, has a long history, but it is  mostly focussed on the case where the operators $A(t)$ all have purely discrete spectrum.

Introducing the operators
$\bsH_1=\bsD_\bsA^*\bsD^{}_\bsA$ and $\bsH_2=\bsD_\bsA^{}\bsD_\bsA^*$, we consider 
spectral shift functions, denoted by $\xi(\, \cdot \,; A_+, A_-)$ and $\xi(\, \cdot \, ; \bsH_2,\bsH_1)$ associated with the pairs $(A_+, A_-)$ and $(\bsH_2,\bsH_1)$. 
Under the restrictive hypotheses that $A_+$ is a relatively trace class perturbation of $A_-$,  
a relationship between these spectral shift functions was proved in \cite{GLMST11}, for 
certain operators $A_\pm$ with  essential spectrum, extending a result of Pushnitski \cite{Pu08}. 
Moreover, assuming $A_{\pm}$ to be Fredholm, the value $\xi(0; A_-, A_+)$ was shown to represent the spectral flow along the path 
$\{A(t)\}_{t\in \bbR}$ while that of $\xi(0_+; \bsH_1,\bsH_2)$ yields the Fredholm index of 
$\bsD_\bsA$.
The fact, proved in \cite{GLMST11}, that these values of the two spectral functions
are equal, resolves the index = spectral flow question in this case. 
This relationship between spectral shift functions was generalized to non-Fredholm operators in \cite{CGPST14a} again under 
the  relatively trace class perturbation hypothesis. In this situation it asserts that the Witten index
of $\bsD_\bsA$, denoted by $W_r(\bsD_\bsA^{})$, a substitute for the Fredholm index in the absence of the Fredholm property of $\bsD_\bsA^{}$, is given by 
$$
W_r(\bsD_\bsA^{}) = \Lxi(0_+; \bsH_2, \bsH_1) 
= [\Lxi(0_+; A_+,A_-) + \Lxi(0_-; A_+, A_-)]/2.     
$$
Here one assumes that $\xi(\, \cdot \,; A_-, A_+)$ possesses a right and left Lebesgue point at $0$ denoted by $\Lxi(0_{\pm}; A_+, A_-)$ (and similarly for $\Lxi(0_+; \bsH_2, \bsH_1)$).

When the path $\{A(t)\}_{t \in \bbR}$ consists of differential operators, the relatively trace class perturbation assumption is violated. The simplest assumption that applies (to differential operators 
in 1+1 dimensions) is to admit relatively Hilbert--Schmidt perturbations. This is not just
 an incremental improvement. In fact, the method we employ here to make this extension is of interest in any dimension. Moreover we consider $A_\pm$ which are not necessarily Fredholm and we establish that the relationships between the two spectral shift functions found in
{\bf all} of the previous papers \cite{CGPST14a} ,\cite{GLMST11}, and  \cite{Pu08}, can be proved, even in the non-Fredholm case. The significance of our new methods is that, besides being simpler,
 they also allow a wide class of examples such as pseudodifferential operators in higher dimensions. Most importantly, we prove  the above formula for the Witten index in the most general circumstances to date.  
\end{abstract}

\maketitle

\newpage 

{\scriptsize{\tableofcontents}}

\section{Introduction}  \lb{s1}

Typical Hamiltonians in quantum mechanical models have some essential spectrum and this is true more generally for differential operators on non-compact manifolds.
Much less is known about Fredholm theory and spectral flow in these situations.  Our objective in this paper is to investigate operators for which zero is in the essential spectrum
and for which Fredholm theory is not applicable. 

This paper is motivated by \cite{CGPST14a} where results on an index theory for certain non-Fredholm operators are described using the model operator formalism in \cite{GLMST11}. The latter paper was inspired by \cite{Pu08} which, in turn, was motivated by \cite{RS95}, where 
the relationship between the Fredholm index and spectral flow for operators with discrete spectrum is studied. The model operators considered there provide prototypes  for more complex situations. They arise in connection with investigations of the Maslov index, Morse theory, Floer homology, Sturm oscillation theory, etc. 

The principal aim in \cite{Pu08} and \cite{GLMST11} was to extend the discrete spectrum results of \cite{RS95}, relating the Fredholm index and spectral flow,  to a relatively trace class 
perturbation theory approach, permitting essential spectra. However the relatively trace class  assumption rules out standard differential operators such as Dirac-type operators
and thus, in order to incorporate this important class of examples, we need a more general framework.  In this paper we introduce a new approach that enables us to handle relatively 
Hilbert--Schmidt perturbations.  This  improvement  incorporates 1+1 dimensional differential operators while the methods introduced here are also applicable in more general situations.  Moreover we prove the main result of \cite{GLMST11} by a shorter and simpler method at the same time as generalizing it substantially so that it applies to non-Fredholm 
operators.

To introduce the model, let $\{A(t)\}_{t\in\bbR}$ be a family of self-adjoint operators in 
the complex, separable Hilbert space $\cH$, subject to the 
assumption that self-adjoint  limiting operators
\begin{equation} 
  A_+=\lim_{t\to+\infty}A(t), \quad A_-=\lim_{t\to-\infty}A(t)    \lb{1.2}
\end{equation} 
exist in $\cH$ in the norm resolvent sense. We denote by $\bsA$ the operator in 
$L^2(\bbR;\cH)$ defined by  
\begin{align}
&(\bsA f)(t) = A(t) f(t) \, \text{ for a.e.\ $t\in\bbR$,}   \no \\
& f \in \dom(\bsA) = \bigg\{g \in L^2(\bbR;\cH) \,\bigg|\,
g(t)\in \dom(A(t)) \text{ for a.e.\ } t\in\bbR;     \lb{1.1} \\
& \quad t \mapsto A(t)g(t) \text{ is (weakly) measurable;} \,  
\int_{\bbR} dt \, \|A(t) g(t)\|_{\cH}^2 < \infty\bigg\}.   \no 
\end{align}

Next, we introduce the model operator 
\begin{equation}
\bsD_\bsA^{} = \f{d}{dt} + \bsA,
\quad \dom(\bsD_\bsA^{})= \dom(d/dt) \cap \dom(\bsA_-),     \lb{1.5}   \\
\end{equation}
and the associated nonnegative, self-adjoint operators 
\begin{equation}
\bsH_1 = \bsD_\bsA^* \bsD_\bsA^{}, \quad \bsH_2 = \bsD_\bsA^{} \bsD_\bsA^*,   
\lb{1.6}
\end{equation}
in $L^2(\bbR;\cH)$. (Here $\bsA_-$  in $L^2(\bbR;\cH)$ represents the self-adjoint constant fiber operator defined according to \eqref{1.1}, with $A(t)$ replaced by the asymptote $A_-$). 

Assuming that $A_-$ and $A_+$ are boundedly invertible, we also recall (cf.\ \cite{GLMST11}) that
$\bsD_\bsA^{}$ is a Fredholm operator in $L^2(\bbR;\cH)$. 
Under the additional relatively trace class assumption, that is, 
$(A_+-A_-)\big(A_-^2 + I_{\cH}\big)^{-1/2}$ is trace class, it is shown in 
\cite{GLMST11} (and earlier in \cite{Pu08} under a more stringent set of hypotheses on the family 
$A(\cdot)$), that the Fredholm index of $\bsD_\bsA^{}$ may then be computed as follows,   
\begin{align}
\ind (\bsD_\bsA^{}) & = \xi(0_+; \bsH_2, \bsH_1)   
 = \xi(0; A_+, A_-).   \lb{1.10} 
\end{align}
Here $\xi(\, \cdot \, ; S_2,S_1)$ denotes the spectral shift function for 
the pair of self-adjoint operators $(S_2, S_1)$. Whenever $S_j$, $j=1,2$, are bounded from below, 
we adhere to the normalization
\begin{equation}
\xi(\lambda ; S_2,S_1) = 0 \, \text{ for } \, \lambda < \inf(\sigma(S_1) \cup \sigma(S_2)),  
\end{equation}  
in particular, $\xi(\lambda; \bsH_2, \bsH_1) = 0$, $\lambda < 0$.

The new direction developed in \cite{CGPST14a} focuses on the model operator 
$\bsD_\bsA^{}$ in $L^2(\bbR;\cH)$ whenever the latter ceases to be Fredholm. 
First, we recall the definition of the Witten index as studied 
in \cite{BGGSS87} and \cite{GS88}: 
\begin{equation}
W_r(\bsD_\bsA^{}) = \lim_{\lambda \uparrow 0} (- \lambda) 
\tr_{L^2(\bbR; \cH)}\big((\bsH_1 - \lambda \bsI)^{-1} - (\bsH_2 - \lambda \bsI)^{-1}\big),     \lb{1.13a}
\end{equation} 
whenever the limit exists (where $\tr_{\cK}(\cdot)$ abbreviates the trace in the Hilbert space $\cK$). Here, the subscript ``r'' refers to the resolvent regularization used; other regularizations, for instance, semigroup (heat kernel) based ones, are 
possible (cf., \cite{CGPST14a}).  

If $\bsD_\bsA^{}$ is Fredholm (and of course the necessary trace class conditions in \eqref{1.13a} are satisfied), one has consistency with the Fredholm index
of $\bsD_\bsA^{}$. In addition, under appropriate spectral assumptions  the following connection between Fredholm, respectively, Witten indices and the underlying spectral shift function applies
\begin{equation}
W_r(\bsD_\bsA^{}) = \xi(0_+; \bsH_2, \bsH_1).    \lb{1.13c}
\end{equation} 
Most importantly, $W_r(\bsD_\bsA^{})$ exhibits invariance properties under 
additive, relatively trace class perturbations (apart from some additional technical 
hypotheses). This is sometimes dubbed topological invariance of the Witten index 
in the pertinent literature (see, e.g., \cite{BGGSS87}, \cite{Ca78}, \cite{GLMST11}, \cite{GS88}, and the references therein).

Originally, index regularizations such as \eqref{1.13a} were studied in the context of 
supersymmetric quantum mechanics in the physics literature in the 1970's and 1980's, see, \cite{CGPST14a} for details.

The results of \cite{CGPST14a} for the
specific model operator $\bsD_\bsA^{}$ in $L^2(\bbR;\cH)$ are as follows. 
Assuming that $0$ is a right 
and a left Lebesgue point of $\xi(\,\cdot\,\, ; A_+, A_-)$ (denoted by $\Lxi(0_+; A_+,A_-)$ 
and $\Lxi(0_-; A_+, A_-)$, respectively), then it is 
also a right Lebesgue point of $\xi(\,\cdot\,\, ; \bsH_2, \bsH_1)$ (denoted by 
$\Lxi(0_+; \bsH_2, \bsH_1)$). Under this right/left Lebesgue point 
assumption on $\xi(\,\cdot\,\, ; A_+, A_-)$, the analog of \eqref{1.5} and \eqref{1.6} in the 
general case where $\bsD_\bsA^{}$ ceases to be Fredholm, the  principal new 
result of \cite{CGPST14a}, then reads as follows, 
\begin{align} 
W_r(\bsD_\bsA^{}) & = \Lxi(0_+; \bsH_2, \bsH_1) = [\Lxi(0_+; A_+,A_-) + \Lxi(0_-; A_+, A_-)]/2.    \lb{1.16}
\end{align}

Recalling our earlier mention of Dirac operators, let $A_-$ be the flat space Dirac operator on spinor valued functions on
$\R^n$ and perturb it by a multiplication operator by a function $\phi:\R^n\to \R$ (matrix-valued functions also
provide examples) to define the operator $A_+=A_-+\phi$. It follows from the discussion in Remark~(c) of \cite[Chapter~4]{Si05}
that, for $\phi$ of sufficiently rapid decay at $\pm\infty$, we can show that
$(A_+-A_-)(A_-^2 + I_{\cH})^{-s/2}$ is trace class for $s>n$ but for no lesser 
value of $s$. Thus, even in one dimension the relatively trace class perturbation assumption is violated for geometric examples based on Dirac-type operators.

{\it Our main objective in this paper is to extend the results described above to situations in which
the relatively trace class perturbation assumption no longer holds, replacing it by 
Hypothesis \ref{h3.4} below.}
The new technique described here is an approximation argument that amounts,
for differential operators, to using pseudodifferential approximating operators.  Then, for the approximants, the relatively trace class
perturbation condition is restored and the results of our earlier papers on spectral shift functions (\cite{CGPST14}, \cite{CGPST14a}) are available
for use.  So we obtain strong information on the  spectral shift functions for the approximants. Then we find that they may be shown to converge to the spectral shift functions for the original operators,
as the pseudodifferential approximating operators converge in an appropriate (strong resp., norm resolvent) sense. This leads to one of our main results, \eqref{1.16} under the most general 
Hypothesis \eqref{h3.4} to date, which generalizes the main theorem of \cite{GLMST11} 
and the subsequent \cite{CGPST14a}.

While the strategy of this paper applies in higher dimensions (where the relatively Hilbert--Schmidt condition is replaced by a relatively Schatten class condition) further additional ideas are needed to make our approximation scheme work there.   Examples in  \cite{CGK15} illustrate some of the issues.  This matter is  a part of ongoing investigations. 
We remark that the topological meaning of the Witten index is explored in \cite{GS88} and \cite{CK14} while its geometric significance is still under investigation.

\section{The Strategy Employed} \lb{s2} 

In this section we briefly outline the principal new strategy employed in this paper that 
permits us to circumvent the relative trace class Hypothesis~2.1 used in 
\cite{GLMST11} and \cite{CGPST14a}.
Throughout this section we assume that Hypothesis \ref{h3.4} is valid. 

$\mathbf{I}$. Consider the family of 
self-adjoint operators $A(t)$, $t\in \bbR$, with asymptotes $A_{\pm}$ as well as 
$B(t)$, $t \in \bbR$, and $B_{\pm}$, such that 
\begin{equation}
A(t) = A_- + B(t), \quad t \in \bbR. 
\end{equation} 
Introduce $\bsD_\bsA^{} = \f{d}{dt} + \bsA$ in $L^2(\bbR; \cH)$ as in \eqref{1.5}, 
and define $\bsA_-$  in $L^2(\bbR; \cH)$ as  
the self-adjoint (constant fiber) operator defined in \eqref{2.DA-}. 
Next, introduce the operators $\bsH_j$, $j=1,2$, in $L^2(\bbR; \cH)$ by
$\bsH_1 = \bsD_{\bsA}^{*} \bsD_{\bsA}^{}$, $\bsH_2 = \bsD_{\bsA}^{} \bsD_{\bsA}^{*}$ 
as in \eqref{1.6}. In addition, assuming continuous differentiability of the family $B(\cdot)$ and 
introducing $\bsB$ and $\bsB^{\prime}$ in terms of the bounded operator families 
$B(t)$, $B'(t)$, $t \in \bbR$, in analogy to \eqref{1.S}, one can decompose 
$\bsH_j$, $j=1,2$, as follows:
\begin{equation}
\bsH_j = \bsH_0 + \bsB \bsA_- + \bsA_- \bsB + \bsB^2 + (-1)^j \bsB^{\prime}, \quad j =1,2, 
\end{equation}
with 
\begin{equation}
\bsA = \bsA_- + \bsB, \quad \dom(\bsA) = \dom(\bsA_-).  
\end{equation} 

$\mathbf{II}$. Use the approximation
\begin{equation}
A_n(t) = A_- + \chi_n(A_-) B(t) \chi_n(A_-), \quad n \in \bbN, \; t \in \bbR, 
\end{equation}
with  
\begin{equation}
\chi_n(\nu) = \f{n}{(\nu^2 + n^2)^{1/2}}, \quad \nu \in \bbR, \; n \in \bbN, 
\end{equation} 
such that $\slim_{n \to \infty} \chi_n(A_-) = I$.

Given our assumptions on $B(\cdot)$ (cf.\ Hypothesis \ref{h3.4}) one infers that 
\begin{align}
& A_{-,n} = A_-, \quad 
A_{+,n} = A_- + \chi_n(A_-) B_+ \chi_n(A_-), \quad n \in \bbN,   \\
& A_{+,n} - A_- = \chi_n(A_-) B_+ \chi_n(A_-) \in \cB_1(\cH), \quad n \in \bbN,   \\  
& A_n'(t) = B_n'(t) = \chi_n(A_-) B'(t) \chi_n(A_-) \in \cB_1\big(\cH\big), 
\quad n \in \bbN, \; t \in \bbR,  \\
& \int_{\bbR} dt \, \|A_n'(t)\|_{\cB_1(\cH)} < \infty.     \lb{} 
\end{align}
(We denote by $\cB_p(\cH)$ the standard $\ell^p$-based trace ideals, 
$p \geq 1$.) Thus, one also obtains,
\begin{align}
& \bsH_{j,n} = \bsH_0 + \bsB_n \bsA_- + \bsA_- \bsB_n 
+ \bsB_n^2 + (-1)^j \bsB_n^{\prime}, \quad 
\dom(\bsH_{j,n}) = \dom(\bsH_0),   \no \\
& \hspace*{8.5cm}  n \in \bbN, \; j=1,2, 
\end{align}
with
\begin{equation}
\bsB_n = \chi_n(\bsA_-) \bsB \chi_n(\bsA_-), \quad 
\bsB_n^{\prime} = \chi_n(\bsA_-) \bsB^{\prime} \chi_n(\bsA_-), \quad n \in \bbN. 
\end{equation} 

$\mathbf{III}$. As a consequence of step $\mathbf{II}$, the spectral shift functions 
$\xi(\, \cdot \, ; A_{+,n}, A_-)$ and $\xi(\, \cdot \,; \bsH_{2,n}, \bsH_{1,n})$, $n \in \bbN$, 
exist and are uniquely determined by 
\begin{equation}
\xi(\, \cdot \, ; A_{+,n}, A_-) \in L^1(\bbR; d\nu), \quad 
\xi(\lambda; \bsH_{2,n}, \bsH_{1,n}) = 0, \;  \lambda < 0, \quad n \in \bbN. 
\end{equation}
Moreover, employing \cite{CGPST14} or \cite{Pu08}, one obtains the approximate trace formula,
\begin{equation}
\int_{[0,\infty)} \f{\xi(\lambda; \bsH_{2,n}, \bsH_{1,n}) \, d\lambda}{(\lambda - z)^2} 
= \f{1}{2} \int_{\bbR} \f{\xi(\nu; A_{+,n}, A_-) \, d\nu}{(\nu^2 -z)^{3/2}}, \quad 
n \in \bbN, \; z \in \bbC \backslash [0,\infty).       \lb{1}
\end{equation}

$\mathbf{IV}$. {\bf The main result}. Now we take the limits $n \to \infty$ in \eqref{1}. We use the trace norm convergence result in Theorem \ref{t3.7} in combination with some Fredholm determinant facts to control the limit $n \to \infty$ of the left-hand and right-hand side of \eqref{1} to arrive at
\begin{equation}
\int_{[0,\infty)} \f{\xi(\lambda; \bsH_2, \bsH_1) \, d\lambda}{(\lambda - z)^2} 
= \f{1}{2} \int_{\bbR} \f{\xi(\nu; A_+, A_-) \, d\nu}{(\nu^2 -z)^{3/2}}, \quad 
z \in \bbC \backslash [0,\infty).      \lb{3}
\end{equation}
Relation \eqref{3} combined with a Stieltjes inversion argument then  implies the main formula of the paper (a Pushnitski-type relation between spectral shift functions):
\begin{equation}
\xi(\lambda; \bsH_2, \bsH_1) = \f{1}{\pi} \int_{- \lambda^{1/2}}^{\lambda^{1/2}} 
\f{\xi(\nu; A_+, A_-) \, d \nu}{(\lambda - \nu^2)^{1/2}} \, 
\text{ for a.e.~$\lambda > 0$.}    \lb{4} 
\end{equation}

As a result, one of the principal theorems proven in this paper then proves the equalities in 
\eqref{1.16}. In particular, Equations \eqref{3}, \eqref{4}, and \eqref{1.16} represent the analog 
of the principal results in \cite{GLMST11} and \cite{CGPST14a} for the model operator 
$\bsD_\bsA^{}$, but now under considerably more general hypotheses. 

\smallskip

{\bf Notation.} We briefly summarize some of the notation used in this paper: 
Let $\cH$ be a separable complex Hilbert space, $(\cdot,\cdot)_{\cH}$ the scalar product in $\cH$
(linear in the second argument), and $I_{\cH}$ the identity operator in $\cH$.

Next, if $T$ is a linear operator mapping (a subspace of) a Hilbert space into another, then 
$\dom(T)$ and $\ker(T)$ denote the domain and kernel (i.e., null space) of $T$. 
The closure of a closable operator $S$ is denoted by $\ol S$. 
The spectrum, essential spectrum, discrete spectrum, point spectrum, and resolvent set 
of a closed linear operator in a Hilbert space will be denoted by $\sigma(\cdot)$, 
$\sigma_{\rm ess}(\cdot)$, $\sigma_{\rm d}(\cdot)$, $\sigma_{\rm p}(\cdot)$, and 
$\rho(\cdot)$, respectively. 

The convergence of bounded operators in the strong operator topology (i.e., pointwise limits) will be denoted by $\slim$, similarly, norm limits of bounded operators are denoted by $\nlim$. 


The Banach spaces of bounded and compact linear operators on a separable complex Hilbert space $\cH$ are denoted by $\cB(\cH)$ and $\cB_\infty(\cH)$, respectively; the corresponding $\ell^p$-based trace ideals will be denoted by $\cB_p (\cH)$, $p>0$. Moreover, ${\det}_{\cH}(I_\cK-A)$, and $\tr_{\cH}(A)$ denote the standard Fredholm determinant and the corresponding trace of a trace class operator $A\in\cB_1(\cH)$. 

Linear operators in the Hilbert space $L^2(\bbR; dt; \cH)$, in short, $L^2(\bbR; \cH)$, will be denoted by calligraphic boldface symbols of the type $\bsT$, to distinguish them from 
operators $T$ in $\cH$. In particular, operators denoted by 
$\bsT$ in the Hilbert space $L^2(\bbR;\cH)$ typically represent operators associated with a 
family of operators $\{T(t)\}_{t\in\bbR}$ in $\cH$, defined by
\begin{align}
&(\bsT f)(t) = T(t) f(t) \, \text{ for a.e.\ $t\in\bbR$,}    \no \\
& f \in \dom(\bsT) = \bigg\{g \in L^2(\bbR;\cH) \,\bigg|\,
g(t)\in \dom(T(t)) \text{ for a.e.\ } t\in\bbR;    \lb{1.S}  \\
& \quad t \mapsto T(t)g(t) \text{ is (weakly) measurable;} \, 
\int_{\bbR} dt \, \|T(t) g(t)\|_{\cH}^2 <  \infty\bigg\}.   \no
\end{align}
In the special case, where $\{T(t)\}$ is a family of bounded operators on $\cH$ with 
$\sup_{t\in\bbR}\|T(t)\|_{\cB(\cH)}<\infty$, the associated operator $\bsT$ is a bounded operator on $L^2(\bbR;\cH)$ with $\|\bsT\|_{\cB(L^2(\bbR;\cH))} = \sup_{t\in\bbR}\|T(t)\|_{\cB(\cH)}$.

For brevity we will abbreviate $I:= I_{\cH}$ and $\bsI := I_{L^2(\bbR; \cH)}$.

\section{The Basic Setup} \lb{s3}

In this section we set the stage for an extension of 
\cite{GLMST11} that circumvents the apparently fundamental relative trace class condition 
in Hypothesis~2.1 in \cite{GLMST11} and \cite{CGPST14a}.

We start by collecting our principal assumptions:

\begin{hypothesis} \lb{h3.1} Suppose $\cH$ is a complex, separable Hilbert space. \\
$(i)$ Assume $A_-$ is self-adjoint on $\dom(A_-) \subseteq \cH$. \\
$(ii)$ Suppose we have a family of bounded operators 
$\{B(t)\}_{t \in \bbR} \subset \cB(\cH)$, continuously differentiable in norm on $\bbR$ 
such that 
\begin{equation}
\|B'(\cdot)\|_{\cB(\cH)} \in L^1(\bbR; dt).     \lb{intB'}
\end{equation}
\end{hypothesis}

Given Hypothesis \ref{h3.1}, we introduce the family of self-adjoint operators 
$A(t)$, $t \in \bbR$, in $\cH$, by 
\begin{equation}
A(t) = A_- + B(t), \quad \dom(A(t)) = \dom(A_-), \quad t \in \bbR.
\end{equation}
Moreover, writing
\begin{equation}
B(t) = B(t_0) + \int_{t_0}^t ds \, B'(s), \quad t, t_0 \in \bbR, 
\end{equation}
one infers that the self-adjoint asymptotes 
\begin{equation} 
\nlim_{t \to \pm \infty} B(t) := B_{\pm} \in \cB(\cH)     \lb{limB}
\end{equation}  
exist. In particular, we will make the choice
\begin{equation}
B_- = 0
\end{equation}
in the following and also introduce the asymptote, 
\begin{equation}
A_+ = A_- + B_+, \quad \dom(A_+) = \dom(A_-).    \lb{3.4} 
\end{equation}
Assumption \eqref{intB'} also yields,
\begin{equation}
\sup_{t \in \bbR} \|B(t)\|_{\cB(\cH)} \leq \int_{\bbR} dt \, \|B'(t)\|_{\cB(\cH)}< \infty.     \lb{supB} 
\end{equation}
A simple application of the resolvent identity yields (with $t \in \bbR$, 
$z \in \bbC \backslash \bbR$)
\begin{align}
& (A(t) - z I)^{-1} = (A_{\pm} - z I)^{-1} 
- (A(t) - z I)^{-1} [B(t) - B_{\pm}] (A_{\pm} - z I)^{-1},   \\
& \big\|(A(t) - z I)^{-1} - (A_{\pm} - z I)^{-1}\big\|_{\cB(\cH)} \leq 
|\Im(z)|^{-2} \|B(t) - B_{\pm}\|_{\cB(\cH)},
\end{align}
and hence proves that 
\begin{equation}
\nlim_{t \to \pm \infty} (A(t) - z I)^{-1} = (A_{\pm} - z I)^{-1}, \quad 
z \in \bbC \backslash \bbR. 
\end{equation}

At this point we need to introduce additional hypotheses to those in 
Hypothesis \ref{h3.1}. These additional requirements can be accommodated for differential operators in low dimensions.  We know that a weakening of the next hypothesis is needed for higher-dimensional geometric examples such as those provided by Dirac-type operators. 

\begin{hypothesis} \lb{h3.2} In addition to Hypothesis \ref{h3.1}, assume the following
conditions on $B_+$, $B(t)$, $t \in \bbR$: \\
$(i)$ Suppose that $|B_+|^{1/2}(A_- -z_0 I)^{-1} \in \cB_2(\cH)$ for some $($and hence 
for all\,$)$ $z_0 \in \rho(A_-)$. \\
$(ii)$ Assume that $\sup_{t \in \bbR} \|B'(t)\|_{\cB(\cH)} < \infty$.
\end{hypothesis}

\begin{remark} For Dirac-type operators, the Hilbert--Schmidt condition in Hypothesis \ref{h3.2} is directly tied to the fact that we are considering differences of resolvents in \eqref{3.32}.  We know from \cite{CGK15} that we need to consider differences of higher powers of resolvents in 
higher-dimensional examples whose treatment is deferred to future investigations.
\hfill$\diamond$
\end{remark}

Assuming Hypothesis \ref{h3.2} in the following, the resolvent identity 
\begin{equation}
(A_+ - z I)^{-1} = (A_- - z I)^{-1} 
- (A_+ - z I)^{-1} B_+ (A_- - z I)^{-1},  \quad z \in \bbC \backslash \bbR, 
\end{equation}
combined with Hypothesis \ref{h3.2}\,$(i)$ yields 
\begin{align}
& (A_+ - z I)^{-1} - (A_- - z I)^{-1} = \big[(A_- - \ol{z} I) (A_+ - \ol{z} I)^{-1}\big]^*
\big[|B_+|^{1/2} (A_- - \ol{z} I)^{-1}\big]^*   \no \\ 
& \quad \times \sgn(B_+) \big[|B_+|^{1/2} (A_- - z I)^{-1}\big], 
\quad z \in \bbC \backslash \bbR, 
\end{align}
and hence 
\begin{equation}
\big[(A_+ - z I)^{-1} - (A_- - z I)^{-1}\big] \in \cB_1\big(\cH\big), 
\quad z \in \bbC \backslash \bbR.      \lb{3.13} 
\end{equation} 

Next, we turn to the pair $(\bsH_2, \bsH_1)$, assuming Hypothesis \ref{h3.2}: 
First, we recall that $\bsA$, $\bsB, \bsA '=\bsB', $ 
are defined in terms of the families $A(t)$, $B(t)$, and $B'(t)$, $t\in\bbR$, as in \eqref{1.S}. In addition, $\bsA_-$  in $L^2(\bbR;\cH)$ represents 
the self-adjoint (constant fiber) operator defined by 
\begin{align}
&(\bsA_- f)(t) = A_- f(t) \, \text{ for a.e.\ $t\in\bbR$,}   \no \\
& f \in \dom(\bsA_-) = \bigg\{g \in L^2(\bbR;\cH) \,\bigg|\,
g(t)\in \dom(A_-) \text{ for a.e.\ } t\in\bbR,    \no \\
& \quad t \mapsto A_- g(t) \text{ is (weakly) measurable,} \,  
\int_{\bbR} dt \, \|A_- g(t)\|_{\cH}^2 < \infty\bigg\}.    \lb{2.DA-}
\end{align} 
Now we introduce the operator $\bsD_\bsA^{}$ in $L^2(\bbR;\cH)$ by 
\begin{equation}
\bsD_\bsA^{} = \f{d}{dt} + \bsA,
\quad \dom(\bsD_\bsA^{})= W^{1,2}(\bbR; \cH) \cap \dom(\bsA_-),   \lb{2.DA}
\end{equation} 
where we used that \eqref{supB} implies
\begin{equation}
\|\bsB\|_{\cB(L^2(\bbR; \cH))} = \sup_{t \in \bbR} \|B(t)\|_{\cB(\cH)} < \infty
\end{equation}
and hence
\begin{equation}
\bsA = \bsA_- + \bsB, \quad \dom(\bsA) = \dom(\bsA_-). 
\end{equation}

Here the operator $d/dt$ in $L^2(\bbR;\cH)$  is defined by 
\begin{align}
& \bigg(\f{d}{dt}f\bigg)(t) = f'(t) \, \text{ for a.e.\ $t\in\bbR$,}    \no \\
& \, f \in \dom(d/dt) = \big\{g \in L^2(\bbR;\cH) \, \big|\,
g \in AC_{\loc}\big(\bbR; \cH\big), \, g' \in L^2(\bbR;\cH)\big\}   \no \\
& \hspace*{2.3cm} = W^{1,2} \big(\bbR; \cH\big).     \label{2.ddt} 
\end{align} 
Clearly (cf.\ \cite[Lemma 4.4]{GLMST11}, which also holds for our more general setting), 
$\bsD_\bsA^{}$ is densely defined and closed in $L^2\big(\bbR; \cH\big)$  and the adjoint 
operator $\bsD_\bsA^*$ of $\bsD_\bsA^{}$ is given by
\begin{equation}
\bsD_\bsA^*=- \f{d}{dt} + \bsA, \quad
\dom(\bsD_\bsA^*) = W^{1,2}(\bbR; \cH) \cap \dom(\bsA_-).   
\end{equation}

This enables one to introduce the nonnegative, self-adjoint operators 
$\bsH_j$, $j=1,2$, in $L^2(\bbR;\cH)$ by
\begin{equation}
\bsH_1 = \bsD_{\bsA}^{*} \bsD_{\bsA}^{}, \quad 
\bsH_2 = \bsD_{\bsA}^{} \bsD_{\bsA}^{*}.
\end{equation} 

In order to effectively describe the domains of $\bsH_j$, $j=1,2$, we need to 
decompose the latter as follows:  First, we introduce $\bsB^{\prime}$ in terms of 
the bounded operator families $B'(t)$, $t \in \bbR$, in analogy to \eqref{1.S}, 
and observe that Hypothesis \ref{h3.2}\,$(ii)$ implies that
\begin{equation}
\|\bsB'\|_{\cB(L^2(\bbR; \cH))} = \sup_{t \in \bbR} \|B'(t)\|_{\cB(\cH)} < \infty.  \lb{B'bd}
\end{equation}
Next, we strengthen our hypothesis on $B(t)$, $t \in \bbR$, as follows. Introduce 
$\bsH_0$ in $L^2(\bbR; \cH)$ by  
\begin{equation}
\bsH_0 = - \f{d^2}{dt^2} + \bsA_-^2, \quad \dom(\bsH_0) 
= W^{2,2}(\bbR; \cH) \cap \dom\big(\bsA_-^2\big).     \lb{H0}
\end{equation}
Then $\bsH_0$ is self-adjoint by Theorem VIII.33 of \cite{RS80}.

Again we need to make some additional hypotheses motivated by differential operator
examples.
\begin{hypothesis} \lb{h3.3} In addition to Hypotheses \ref{h3.2}, 
assume the following conditions on $B(t)$, $t \in \bbR$: \\
$(i)$ Suppose that $\bsA_- \bsB$ is bounded with respect to $\bsH_0$ with bound strictly 
less than one, that is, there exists $a \in (0, 1)$ and $b \in (0,\infty)$ such that 
\begin{equation}
\|\bsA_- \bsB f\|_{L^2(\bbR; \cH)} \leq a \|\bsH_0 f\|_{L^2(\bbR; \cH)} 
+ b \|f\|_{L^2(\bbR; \cH)}, \quad f \in \dom(\bsH_0).    \lb{3.24} 
\end{equation} 
$(ii)$ Assume that for some $($and hence for all\,$)$ $z_0 \in \bbC \backslash \bbR$, 
\begin{equation}
|\bsB'|^{1/2} (\bsH_0 -z_0 \, \bsI)^{-1} \in \cB_2\big(L^2(\bbR; \cH)\big)  \lb{3.25}
\end{equation}
\end{hypothesis}

\begin{remark} \lb{r3.5} 
While it is clear that $\bsB \bsA_-$ is infinitesimally bounded with respect to $\bsH_0$, 
to prove this it suffices to note that 
\begin{align}
\begin{split} 
& \big\|\bsB \bsA_- (\bsH_0 - z \, \bsI)^{-1}\big\|_{\cB(L^2(\bbR; \cH))} \leq |\Im(z)|^{-1/2}
\big\|\bsB\big\|_{\cB(L^2(\bbR; \cH))}      \lb{BA-} \\
& \quad \times \big\|\bsA_- (\bsH_0 - z \, \bsI)^{-1/2}\big\|_{\cB(L^2(\bbR; \cH))}, 
\quad z \in \bbC \backslash [0,\infty),   
\end{split} 
\end{align}
it is not obvious that $\bsA_- \bsB$ is bounded with respect to $\bsH_0$. For later purpose 
we note that \eqref{3.24} implies the existence of $a' \in (a,1)$ such that 
\begin{equation}
\big\|\bsA_- \bsB (\bsH_0 - z \, \bsI)^{-1}\big\|_{\cB(L^2(\bbR; \cH))} \leq a' < 1   \lb{A-BH-0}
\end{equation} 
for $0 < |\Im(z)|$ sufficiently large. ${}$ \hfill $\diamond$
\end{remark}

Assuming Hypothesis \ref{h3.3} in the following, \eqref{3.24} combined with \eqref{BA-} imply that the operator $ \bsB \bsA_- + \bsA_- \bsB$ is $\bsH_0$-bounded with bound less than one, and therefore, by \cite[Theorem~VI.4.3]{Ka80} the following decomposition of the operators $\bsH_j$, $j=1,2$ holds 
\begin{align}
& \bsH_j = \f{d^2}{dt^2} + \bsA^2 + (-1)^j \bsA^{\prime}    \no \\
& \hspace*{5.5mm} = \bsH_0 + \bsB \bsA_- + \bsA_- \bsB + \bsB^2 + (-1)^j \bsB^{\prime},  
\lb{3.27} \\
& \dom(\bsH_j) = \dom(\bsH_0), \quad j =1,2.      \no 
\end{align} 
In addition, 
\begin{align}
\begin{split} 
& (\bsH_2 - z \, \bsI)^{-1} - (\bsH_1 - z \, \bsI)^{-1} = 
- (\bsH_1 - z \, \bsI)^{-1} [2 \bsB'] (\bsH_2 - z \, \bsI)^{-1},   \lb{3.31} \\
& \quad = -2 \big[|\bsB'|^{1/2} (\bsH_1 - \ol{z} \, \bsI)^{-1}\big]^* \sgn(\bsB') 
|\bsB'|^{1/2} (\bsH_2 - z \, \bsI)^{-1}, \\
& \hspace*{6.25cm} z\in\rho(\bsH_1) \cap \rho(\bsH_2), 
\end{split} 
\end{align}
and using that
\begin{align}
& \ol{(\bsH_1 - z \, \bsI)^{-1} (\bsH_0 - z \, \bsI)} 
= \big[(\bsH_0 - {\ol z} \, \bsI) (\bsH_1 - {\ol z} \, \bsI)^{-1}\big]^*  
\in \cB\big(L^2(\bbR; \cH)\big), \\
& (\bsH_0 - z \, \bsI) (\bsH_2 - z \, \bsI)^{-1} \in \cB\big(L^2(\bbR; \cH)\big), 
\quad z \in \rho(\bsH_1) \cap \rho(\bsH_2),     \lb{3.31a} 
\end{align} 
and assumption of Hypothesis \ref{h3.3} (ii) 
one concludes that 
\begin{equation}
\big[(\bsH_2 - z \, \bsI)^{-1} - (\bsH_1 - z \, \bsI)^{-1}\big] \in 
\cB_1\big(L^2(\bbR; \cH)\big), \quad z \in \rho(\bsH_1) \cap \rho(\bsH_2).  
\lb{3.32}
\end{equation} 

The fact \eqref{3.32} implies that the spectral shift function 
$\xi(\, \cdot \, ; \bsH_2, \bsH_1)$ for the pair $(\bsH_2, \bsH_1)$ is well-defined, satisfies
\begin{equation}
\xi(\, \cdot \, ; \bsH_2, \bsH_1) \in L^1\big(\bbR; (\lambda^2 + 1)^{-1} d\lambda\big), 
\end{equation} 
and since $\bsH_j\geq 0$, $j=1,2$, one uniquely introduces $\xi(\,\cdot\,; \bsH_2,\bsH_1)$ 
by requiring that
\begin{equation}
\xi(\lambda; \bsH_2,\bsH_1) = 0, \quad \lambda < 0,    \lb{2.46c}
\end{equation}
implying the Krein--Lifshitz trace formula,  
\begin{align}
\begin{split}
\tr_{L^2(\bbR;\cH)} \big((\bsH_2 - z \, \bsI)^{-1} - (\bsH_1 - z \, 
\bsI)^{-1}\big)
= - \int_{[0, \infty)}  \frac{\xi(\lambda; \bsH_2, \bsH_1) \, 
d\lambda}{(\lambda -z)^2},&     \\
z\in\bbC\backslash [0,\infty).&    \lb{3.45}
\end{split} 
\end{align} 

Next, we deviate from the approximation procedure originally employed in 
\cite{GLMST11} and \cite{Pu08}. We now introduce 
\begin{equation}
\chi_n(\nu) = \f{n}{(\nu^2 + n^2)^{1/2}}, \quad \nu \in \bbR, \; n \in \bbN,   \lb{3.chin}
\end{equation} 
and hence obtain  
\begin{equation}
\slim_{n \to \infty} \chi_n(A_-) = I,     \lb{slim} 
\end{equation}
by an elementary application of the spectral theorem for $A_-$. The precise form 
of $\chi_n$ is of course immaterial, we just need property \eqref{slim} and  
property \eqref{intB'n} below. 

We introduce 
\begin{align}
\begin{split} 
& A_n(t) = A_- + \chi_n(A_-) B(t) \chi_n(A_-) = A_- + B_n(t), \\
& \dom(A_n(t)) = \dom(A_-), \quad n \in \bbN, \; t \in \bbR,    \lb{Ant}
\end{split} 
\end{align}
where 
\begin{equation} 
B_n(t) = \chi_n(A_-) B(t) \chi_n(A_-), \quad n \in \bbN, \; t \in \bbR.  
\end{equation}
In addition, we introduce  
\begin{equation}
A_{-,n} = A_-, \quad  \dom(A_{-,n}) = \dom(A_-), \quad n \in \bbN, 
\end{equation} 
and conclude that 
\begin{align} 
& A_{+,n} = A_- + \chi_n(A_-) B_+ \chi_n(A_-), \quad \dom(A_{+,n}) = \dom(A_-), 
\quad n \in \bbN,   \lb{A+n} \\
& A_{+,n} - A_- = \chi_n(A_-) B_+ \chi_n(A_-) \in \cB_1(\cH), \quad n \in \bbN,    \lb{B1An} \\ 
& A_n'(t) = B_n'(t) = \chi_n(A_-) B'(t) \chi_n(A_-) \in \cB_1\big(\cH\big), 
\quad n \in \bbN, \; t \in \bbR.    \lb{Antprime}
\end{align}

As a consequence of \eqref{B1An}, the spectral shift functions 
$\xi(\, \cdot \, ; A_{+,n}, A_-)$, $n \in \bbN$, exist and are uniquely determined by 
\begin{equation}
\xi(\, \cdot \, ; A_{+,n}, A_-) \in L^1(\bbR; d\nu), \quad n \in \bbN. 
\end{equation}

As a preparation to study various limits in Schatten--von Neumann ideals, we now recall 
the following standard convergence property for trace ideals:

\begin{lemma}\lb{l3.6}
Let $p\in[1,\infty)$ and assume that $R,R_n,T,T_n\in\cB(\cH)$, 
$n\in\bbN$, satisfy
$\slim_{n\to\infty}R_n = R$  and $\slim_{n\to \infty}T_n = T$ and that
$S,S_n\in\cB_p(\cH)$, $n\in\bbN$, satisfy 
$\lim_{n\to\infty}\|S_n-S\|_{\cB_p(\cH)}=0$.
Then $\lim_{n\to\infty}\|R_n S_n T_n^\ast - R S T^\ast\|_{\cB_p(\cH)}=0$.
\end{lemma}

As a first of several convergence results we state the following useful fact.

\begin{lemma}\lb{l3.4a}
Assume Hypotheses \ref{h3.1} and  \ref{h3.2}\,$(i)$. Then, 
\begin{equation}
\lim_{n \to \infty} \big\|(A_{+,n} - z I)^{-1} - (A_+ - z I)^{-1}\big\|_{\cB_1(\cH)} = 0, 
\quad z \in \bbC \backslash \bbR.     \lb{limB-1}
\end{equation}
\end{lemma}
\begin{proof}
One writes
\begin{align}
& (A_{+,n} - z I)^{-1} - (A_+ - z I)^{-1} = \big[(A_{+,n} - z I)^{-1} - (A_- - z I)^{-1}\big]    \no \\
& \qquad - \big[(A_+ - z I)^{-1} - (A_- - z I)^{-1}\big]   \no \\
& \quad = - \chi_n(A_-) (A_- - z I)^{-1} B_+  (A_- - z I)^{-1} \chi_n(A_-) 
\big[(A_- - z I)  (A_{+,n} - z I)^{-1}\big]   \no \\
& \qquad +  (A_- - z I)^{-1} B_+  (A_- - z I)^{-1}
\big[ (A_- - z I)  (A_+ - z I)^{-1}\big]    \no \\
& \quad = - \chi_n(A_-) (A_- - z I)^{-1} B_+ (A_- - z I)^{-1} \chi_n(A_-)    \no \\
& \qquad \times \big[I - \chi_n(A_-) B_+ \chi_n(A_-) (A_{+,n} - z I)^{-1}\big]   \no \\
& \qquad +  (A_- - z I)^{-1} B_+ (A_- - z I)^{-1}
\big[I - B_+ (A_+ - z I)^{-1}\big], \quad z \in \bbC \backslash \bbR.  
\end{align}
Thus, relying on Lemma \ref{l3.6} and  \eqref{slim} it suffices to prove that 
\begin{equation}
\slim_{n \to \infty} (A_{+,n} - z I)^{-1} = (A_+ - z I)^{-1}, \quad z \in \bbC \backslash \bbR, 
\end{equation}
but this immediately follows from
\begin{align}
& (A_{+,n} - z I)^{-1} = \big[I + (A_- - z I)^{-1} \chi_n(A_-) B_+ \chi_n(A_-)\big]^{-1} (A_- - z I)^{-1}, \\
& (A_+ - z I)^{-1} = \big[I + (A_- - z I)^{-1} B_+ \big]^{-1} (A_- - z I)^{-1}, 
\end{align}
employing the fact that strong convergence for a sequence of bounded operators is 
equivalent to strong resolvent convergence, initially, for $|\Im(z)|$ suffficiently large, and subsequently, for all $z \in \bbC \backslash \bbR$ by analytic continuation with respect to $z$. 
\end{proof}

Next, going beyond the approximation $A_{+,n}$ of $A_+$, we now introduce the 
following path $\{A_+(s)\}_{s \in [0,1]}$, where 
\begin{align}
& A_+(s) = A_- + \wti \chi_s (A_-) B_+ \wti \chi_s (A_-), 
\quad \dom(A_+(s)) = \dom(A_-),   \quad s \in [0,1],    \lb{B.102} \\
& \wti \chi_s (\nu) = \big[(1-s) \nu^2 + 1\big]^{-1/2}, \quad \nu \in \bbR, \; s \in [0,1],
\end{align}
in particular, 
\begin{equation} 
A_+(0) = A_{+,1} \text{ (cf.\ \eqref{A+n} with $n=1$) and } \, A_+(1) = A_+. 
\end{equation} 
Moreover, in complete analogy to \eqref{limB-1}, the family 
$A_+(s)$ depends continuously on $s \in [0,1]$ with respect to the metric 
\begin{equation}
d(A,A') = \big\|(A - i I)^{-1} - (A' - i I)^{-1}\big\|_{\cB_1(\cH)}    \lb{B.105}
\end{equation}
for $A, A'$ in the set of self-adjoint operators which are resolvent comparable with 
respect to $A_-$ (equivalently, $A_+$), that is, $A, A'$ satisfy 
for some (and hence for all) $\zeta \in \bbC \backslash \bbR$, 
\begin{equation}
\big[(A - \zeta I)^{-1} - (A_- - \zeta I)^{-1}\big], 
\big[(A' - \zeta I)^{-1} - (A_- - \zeta I)^{-1}\big]  \in \cB_1(\cH).  
\end{equation}
Thus, the hypotheses of \cite[Lemma~8.7.5]{Ya92} are satisfied and hence one 
obtains the following result: 

\begin{theorem} \lb{tB.8} 
Assume Hypotheses \ref{h3.1} and  \ref{h3.2}\,$(i)$ and introduce the path 
$A_+(s)$, $s \in [0,1]$, as in 
\eqref{B.102}, with $A_+(0) = A_{+,1}$ $($cf.\ \eqref{A+n} with $n=1$$)$ and $A_+(1) = A_+$. 
Then for each $s \in [0,1]$, there exists a unique spectral shift function 
$\xi(\, \cdot \,; A_+(s), A_-)$ for the pair $(A_+(s), A_-)$, depending continuously on 
$s \in [0,1]$ in the $L^1\big(\bbR; (\nu^2 +1)^{-1} d\nu\big)$-norm, satisfying  
$\xi(\, \cdot \,; A_+(0), A_-) = \xi(\, \cdot \,; A_{+,1}, A_-)$, and $($cf.\ \eqref{B.36a}$)$, 
\begin{align} 
2 i \int_{\bbR} \f{\xi(\lambda; A_+(s),A_-) d \lambda}{\lambda^2 + 1}   
= {\tr}_{\cH} \big(\ln\big(U_+(s)U_-^{-1}\big)\big),     \lb{B.109}
\end{align} 
where
\begin{align}
U_- = (A_- - i I)(A_- + i I)^{-1}, \quad 
U_+(s) = (A_+(s) - i I)(A_+(s) + i I)^{-1}, \; s \in [0,1].
\end{align} 
In addition $($cf.\ \eqref{B.47a}$)$, 
\begin{equation}
\xi(\, \cdot \,; A_+(s),A_-) \in L^1(\bbR; d\nu), \quad s \in [0,1).   \lb{B.108} 
\end{equation}
\end{theorem} 

Thus, observing the equality $\chi_n(\cdot) = \wti \chi_{(1 - n^{-2})}(\cdot)$, Theorem \ref{tB.8} 
implies
\begin{align}
0 &= \lim_{s \uparrow 1} \|\xi(\, \cdot \, ; A_+(s), A_-) 
- \xi(\, \cdot \, ; A_+, A_-)\|_{L^1(\bbR; (\nu^2 + 1)^{-1}d\nu)}   \no \\
&=  \lim_{s \uparrow 1} \|\xi(\, \cdot \, ; A_+(s), A_-)  
- \xi(\, \cdot \, ; A_+(1), A_-)\|_{L^1(\bbR; (\nu^2 + 1)^{-1}d\nu)}    \\ 
&= \lim_{n \to \infty} \|\xi(\, \cdot \, ; A_{+,n}, A_-) 
- \xi(\, \cdot \, ; A_+(1), A_-)\|_{L^1(\bbR; (\nu^2 + 1)^{-1}d\nu)}.  
\end{align}
In particular, a subsequence of $\{\xi(\, \cdot \, ; A_{+,n}, A_-)\}_{n \in \bbN}$ converges 
pointwise a.e.\ to $\xi(\, \cdot \, ; A_+, A_-)$ as $n \to \infty$. Thus, the sequence 
$\xi(\, \cdot \, ; A_{+,n}, A_-) \in L^1(\bbR; d\nu)$, $n \in \bbN$, 
naturally enforces a choice for the open constant inherent in $\xi(\, \cdot \, ; A_+, A_-)$ 
determined by $\xi(\, \cdot \, ; A_{+,1}, A_-) = \xi(\, \cdot \, ; A_+(0), A_-)$, which will henceforth 
be adopted for the remainder of this paper. 

We continue with an elementary but useful consequence of Theorem \ref{tB.8}.

\begin{corollary} \lb{cB.9} 
Assume Hypotheses \ref{h3.1} and  \ref{h3.2}\,$(i)$ and suppose that 
$f \in L^{\infty}(\bbR)$. Then
\begin{equation}
\lim_{n \to \infty} \|\xi(\, \cdot \, ; A_{+,n}, A_-) f 
- \xi(\, \cdot \, ; A_+, A_-) f\|_{L^1(\bbR; (\nu^2 + 1)^{-1}d\nu)} = 0,    \lb{B.118}
\end{equation}
in particular,
\begin{equation}
\lim_{n \to \infty} \int_{\bbR} \xi(\nu; A_{+,n}, A_-) d \nu \, g(\nu) 
= \int_{\bbR} \xi(\nu; A_+, A_-) d \nu \, g(\nu)     \lb{B.119}
\end{equation}
for all $g \in L^{\infty}(\bbR)$ such that 
$\esssup_{\nu \in \bbR} \big|(\nu^2 + 1) g(\nu)\big| < \infty$.
\end{corollary}
\begin{proof}
Relation \eqref{B.118} is clear from Theorem \ref{tB.8} and 
\begin{align}
\begin{split} 
& \|\xi(\, \cdot \, ; A_{+,n}, A_-) f 
- \xi(\, \cdot \, ; A_+, A_-) f\|_{L^1(\bbR; (\nu^2 + 1)^{-1}d\nu)}    \\
& \quad \leq \|f\|_{L^{\infty}(\bbR)} \, \|\xi(\, \cdot \, ; A_{+,n}, A_-) 
- \xi(\, \cdot \, ; A_+, A_-)\|_{L^1(\bbR; (\nu^2 + 1)^{-1}d\nu)},
\end{split} 
\end{align} 
and \eqref{B.119} is obvious from \eqref{B.118} and decomposing the (complex) measures 
\begin{equation} 
\xi(\nu; A_{+,n}, A_-) d \nu \, g(\nu) \, \text{ and } \, \xi(\nu; A_+, A_-) d \nu \, g(\nu)
\end{equation} 
into
\begin{equation}  
(\nu^2 + 1)^{-1} \xi(\nu; A_{+,n}, A_-) d \nu \, (\nu^2 + 1) g(\nu) \, \text{ and } \,  
(\nu^2 + 1)^{-1} \xi(\nu; A_+, A_-) d \nu \, (\nu^2 + 1) g(\nu).
\end{equation}   
\end{proof}

At this point we introduce one more assumption regarding $B_n^{\prime}(t)$, 
$n \in \bbN$, $t \in \bbR$, and for convenience now collect all our hypotheses at one 
place:

\begin{hypothesis} \lb{h3.4} 
Suppose $\cH$ is a complex, separable Hilbert space. \\
$(i)$ Assume $A_-$ is self-adjoint on $\dom(A_-) \subseteq \cH$. \\
$(ii)$ Suppose we have a family of bounded operators 
$\{B(t)\}_{t \in \bbR} \subset \cB(\cH)$, continuously differentiable in norm on $\bbR$ 
such that 
\begin{equation}
\|B'(\cdot)\|_{\cB(\cH)} \in L^1(\bbR; dt).     \lb{intB'a}
\end{equation} 
$(iii)$ Suppose that $|B_+|^{1/2}(A_- -z_0 I)^{-1} \in \cB_2(\cH)$ for some $($and hence 
for all\,$)$ $z_0 \in \rho(A_-)$. $($Here $B_+ = \nlim_{t \to +\infty} B(t)$.$)$ \\
$(iv)$ Assume that $\sup_{t \in \bbR} \|B'(t)\|_{\cB(\cH)} < \infty$. \\  
$(v)$ Suppose that $\bsA_- \bsB$ is bounded with respect to $\bsH_0$ with bound strictly 
less than one, that is, there exists $0 \leq a < 1$ and $b \in (0,\infty)$ such that 
\begin{equation}
\|\bsA_- \bsB f\|_{L^2(\bbR; \cH)} \leq a \|\bsH_0 f\|_{L^2(\bbR; \cH)} 
+ b \|f\|_{L^2(\bbR; \cH)}, \quad f \in \dom(\bsH_0).
\end{equation} 
$(vi)$ Assume that for some $($and hence for all\,$)$ $z_0 \in \bbC \backslash \bbR$, 
\begin{equation}
|\bsB'|^{1/2} (\bsH_0 -z_0 \bsI)^{-1} \in \cB_2\big(L^2(\bbR; \cH)\big)  \lb{3.25a}
\end{equation} 
$(vii)$ Assume that 
\begin{equation}
B_n'(t) \in \cB_1(\cH), \quad \|B_n'(\cdot)\|_{\cB_1(\cH)} \in L^1(\bbR; dt), \quad 
n \in \bbN, \; t \in \bbR.    \lb{intB'n}
\end{equation}
\end{hypothesis}

\begin{remark}
We note that the final two assumptions $(vi)$ and $(vii)$ in Hypothesis \ref{h3.4} can be derived from the stronger condition 
\begin{equation}\label{rel_H-S}
|B'(t)|^{1/2}(|A_-|+I)^{-1}\in\cB_2(\cH), \quad 
\big\||B'(\cdot)|^{1/2}(|A_-|+I)^{-1}\big\|_{\cB_2(\cH)}\in L^2(\bbR; dt).
\end{equation} 

Indeed, repeating the argument in \cite[Lemma~4.6]{GLMST11}, one can obtain the inclusion 
$|\bsB'|^{1/2} (\bsH_0 -z_0 \bsI)^{-1} \in \cB_2\big(L^2(\bbR; \cH)\big)$. In addition, with 
regard to Hypothesis \ref{h3.4}$\,(vii)$, one obtains  
\begin{align}
&\int_\bbR dt \, \|B'_n(t)\|_{\cB_1(\cH)} 
=\int_\bbR dt \, \|\chi_n(A_-)B'(t)\chi_n(A_-)\|_{\cB_1(\cH)}    \no \\
& \quad \leq \int_\bbR dt \, \big\|\overline{\chi_n(A_-)(|A_-|+I)}\big\|_{\cB(\cH)} \, 
\big\||B'(t)|^{1/2}(|A_-|+I)^{-1}\big\|_{\cB_2(\cH)}^2  \\
& \qquad \times 
\big\|(|A_-|+I)\chi_n(A_-)\big\|_{\cB(\cH)},    \no 
\end{align}
and since $\overline{\chi_n(A_-)(|A_-|+I)}, \, (|A_-|+I)\chi_n(A_-) \in \cB(\cH)$, one infers that 
\begin{equation} 
\int_\bbR dt \, \|B'_n(t)\|_{\cB_1(\cH)} 
\leq \int_\bbR dt \, \big\||B'(t)|^{1/2}(|A_-|+I)^{-1}\big\|_{\cB_2(\cH)}^2 < \infty. 
\end{equation} 
${}$ \hfill $\diamond$
\end{remark}

Assuming Hypothesis \ref{h3.4} from now on, one obtains the decompositions,
\begin{align}
& \bsH_{j,n} = \f{d^2}{dt^2} + \bsA_n^2 + (-1)^j \bsA_n^{\prime}    \no \\
& \hspace*{8mm} = \bsH_0 + \bsB_n \bsA_- + \bsA_- \bsB_n 
+ \bsB_n^2 + (-1)^j \bsB_n^{\prime}, \\
& \dom(\bsH_{j,n}) = \dom(\bsH_0) = W^{2,2}(\bbR; \cH),  
\quad n \in \bbN, \;  j =1,2,  \no 
\end{align}
with
\begin{equation}
\bsB_n = \chi_n(\bsA_-) \bsB \chi_n(\bsA_-), \quad 
\bsB_n^{\prime} = \chi_n(\bsA_-) \bsB^{\prime} \chi_n(\bsA_-), \quad n \in \bbN. 
\end{equation} 

\begin{lemma} \lb{l3.5}
Assume Hypothesis \ref{h3.4} and let $z \in \bbC \backslash [0,\infty)$. Then the following assertions hold: \\
$(i)$ The operators $\bsH_{j,n}$ converge to $\bsH_j$, $j=1,2$, 
in the strong resolvent sense, 
\begin{equation}
\slim_{n \to \infty} (\bsH_{j,n}-z\, \bsI)^{-1} 
= ( \bsH_{j}-z\, \bsI)^{-1}, \quad j=1,2.     \lb{limR}
\end{equation}
$(ii)$ The operators 
\begin{align}
& (\bsH_{0}- z \, \bsI) ( \bsH_{j,n}- z \, \bsI)^{-1}, \quad n \in \bbN, \;j=1,2,   \lb{closure} \\
& \ol{( \bsH_{j,n}-z \, \bsI)^{-1}(\bsH_{0}-z \, \bsI)} = 
\big[(\bsH_{0}- {\ol z} \, \bsI) ( \bsH_{j,n}- {\ol z} \, \bsI)^{-1}\big]^*, \quad 
n \in \bbN, \; j=1,2,   \no 
\end{align} 
are uniformly bounded with respect to $n \in \bbN$, that is, there exists 
$C \in (0,\infty)$ such that
\begin{equation}
\big\|(\bsH_{0}- z \, \bsI) ( \bsH_{j,n}- z \, \bsI)^{-1}\big\|_{\cB(L^2(\bbR; \cH))} 
\leq C, \quad n \in \bbN, \; j=1,2.    \lb{bd} 
\end{equation}
In addition, 
\begin{align} 
& \slim_{n \to \infty} \ol{(\bsH_{j,n}-z \, \bsI)^{-1}(\bsH_{0} - z \, \bsI)} 
= \ol{( \bsH_{j}-z \, \bsI)^{-1}(\bsH_{0}-z \, \bsI)}, \quad j=1,2,   \lb{conv1} \\
& \slim_{n \to \infty} (\bsH_{0}-z \, \bsI) (\bsH_{j,n}-z\, \bsI)^{-1} 
= (\bsH_{0}-z \, \bsI)( \bsH_{j}-z\, \bsI)^{-1}, \quad j=1,2.    \lb{conv2}
\end{align}  
\end{lemma}
\begin{proof} Since the proof for the operators $\bsH_{2,n},\bsH_2$ is a verbatim repetition of the proof for $\bsH_{1,n},\bsH_1$, we exclusively focus on the latter. 

\noindent 
$(i)$ By \eqref{3.27} and the analogous equation for $\bsH_{1,n}$, 
the operators $\bsH_1$ and $\bsH_{1,n}$ have a common core 
$\mathrm{dom}(\bsH_1)$. Since 
\begin{equation} 
\bsB^\prime-\bsB^\prime_n=\bsB^\prime-\chi_n(\bsA_-) \bsB^\prime \chi_n(\bsA_-)=(\bsI-\chi_n(\bsA_-))\bsB^\prime+\chi_n(\bsA_-)\bsB^\prime(\bsI-\chi_n(\bsA_-)), 
\end{equation} 
and $\bsB^\prime$ is a bounded operator, the convergence 
\begin{equation}
\slim_{n \to \infty} \bsB^\prime_n = \bsB^\prime    \lb{limBprime}
\end{equation}
holds, employing
\begin{equation} 
\slim_{n \to \infty} \chi_n(\bsA_-) = \bsI    \lb{2.75} 
\end{equation}
(applying the spectral theorem, see also \eqref{slim}). 
Arguing analogously, one also obtains that 
\begin{equation} 
\slim_{n \to \infty} \bsB_n = \bsB.     \lb{limBC}
\end{equation} 
Next, rewriting
\begin{align}
& \bsB^2-\bsB^2_n=\bsB^2-\chi_n(\bsA_-)\bsB\chi_n(\bsA_-)\bsB\chi_n(\bsA_-) 
 \\
& \quad =(\bsI-\chi_n(\bsA_-))\bsB^2+\chi_n(\bsA_-)\bsB\big(\bsB(\bsI-\chi_n(\bsA_-))+(\bsI-\chi_n(\bsA_-))\bsB\chi_n(\bsA_-)\big),   \no 
\end{align}
one also obtains 
\begin{equation} 
\slim_{n \to \infty} \bsB^2_n = \bsB^2.    \lb{limB-2}
\end{equation} 
Thus, it remains to show that for all $f\in\mathrm{dom}(\bsH_1)$, 
$\slim_{n \to \infty} \bsB_n\bsA_- f = \bsB\bsA_-f$ 
and $\slim_{n \to \infty} \bsA_- \bsB_n f = \bsA_- \bsB f$. Indeed, one verifies  
\begin{align}
\bsB\bsA_--\bsB_n\bsA_- &= \bsB \bsA_-  -\chi_n(\bsA_-)\bsB \chi_n(\bsA_-) \bsA_- \\
&=(\bsI-\chi_n(\bsA_-))\bsB\bsA_-+\chi_n(\bsA_-)\bsB(\bsI-\chi_n(\bsA_-))\bsA_-,  
\no 
\end{align} 
and 
\begin{align}
\bsA_- \bsB - \bsA_-  \bsB_n &=  \bsA_- \bsB - \bsA_- \chi_n(\bsA_-)\bsB \chi_n(\bsA_-)\\
&=(\bsI-\chi_n(\bsA_-)) \bsA_- \bsB + \bsA_- \chi_n(\bsA_-)\bsB(\bsI-\chi_n(\bsA_-)),  
\no 
\end{align} 
implying the required convergence. Consequently, 
\begin{equation} 
\slim_{n \to \infty} \bsH_{1,n}f = \bsH_1f, \quad f\in \mathrm{dom}(\bsH_1).  
\end{equation} 
Since $\bsH_{1,n}$ and $\bsH_1$ are self-adjoint operators with a common core, \cite[Theorem~VIII.25]{RS80} (see also \cite[Theorem~9.16]{We80}) implies 
that $\bsH_{1,n}$ converges to $\bsH_1$ in the strong resolvent sense.

\smallskip 
\noindent 
$(ii)$ Fix $z \in \bbC \backslash [0,\infty)$. First, one observes that 
\begin{equation}
\ol{( \bsH_{1,n}-z \, \bsI)^{-1}(\bsH_{0}-z \, \bsI)} = 
\big[(\bsH_{0}- {\ol z} \, \bsI) ( \bsH_{1,n}- {\ol z} \, \bsI)^{-1}\big]^*. 
\end{equation}
Using the standard resolvent identity one obtains
\begin{equation}\label{ssss}
({\bsH}_{1,n}-z \, \bsI)^{-1}-(\bsH_0-z \, \bsI)^{-1}= 
-({\bsH}_{1,n}-z \, \bsI)^{-1}\big[(\bsH_1-\bsH_0)(\bsH_0-z \, \bsI)^{-1}\big], 
\end{equation}
and hence concludes,
\begin{align}
\begin{split} 
& (\bsH_0 - z \, \bsI)(\bsH_{1,n} - z \, \bsI)^{-1}   \\
& \quad = \big[\bsI + \big(\bsA_- \bsB_n + \bsB_n \bsA_- + \bsB_n^2 - \bsB_n'\big)
(\bsH_0 - z \, \bsI)^{-1}\big]^{-1}.     \lb{3.90} 
\end{split}
\end{align} 
Because of \eqref{limBprime} and \eqref{limB-2}, it suffices to focus on the terms 
$\bsB_n \bsA_-$ and $\bsA_- \bsB_n$. As in \eqref{BA-} one estimates 
\begin{align}
\begin{split} 
& \big\|\bsB_n \bsA_- (\bsH_0 - z \, \bsI)^{-1}\big\|_{\cB(L^2(\bbR; \cH))} \leq |\Im(z)|^{-1/2}
\big\|\bsB\big\|_{\cB(L^2(\bbR; \cH))}      \lb{BnA-} \\
& \quad \times \big\|\bsA_- (\bsH_0 - z \, \bsI)^{-1/2}\big\|_{\cB(L^2(\bbR; \cH))}, 
\quad n \in \bbN,   
\end{split} 
\end{align} 
employing $\bsB_n = \chi_n(\bsA_-) \bsB \chi_n(\bsA_-)$,  
$\|\chi_n(\bsA_-)\|_{\cB(L^2(\bbR; \cH))} \leq 1$, $n \in \bbN$, and commutativity of 
$\chi_n(\bsA_-)$ and $(\bsH_0 -z \, \bsI)^{-1}$, that is, 
\begin{equation}
\big[\chi_n(\bsA_-), (\bsH_0 -z \, \bsI)^{-1}\big] = 0, \quad n \in \bbN.    \lb{comm}
\end{equation}
Similarly, utilizing the estimate \eqref{A-BH-0}, one concludes 
\begin{align}
\begin{split} 
& \big\|\bsA_- \bsB_n (\bsH_0 - z \, \bsI)^{-1}\big\|_{\cB(L^2(\bbR; \cH))} \leq 
\big\|\bsA_- \bsB (\bsH_0 - z \, \bsI)^{-1}\big\|_{\cB(L^2(\bbR; \cH))}    \\
& \quad \leq a < 1, \quad n \in \bbN,    \lb{A-B-nH-0}
\end{split} 
\end{align} 
for $0 < |\Im(z)|$ sufficiently large. Thus, choosing $0 < |\Im(z)|$ sufficiently large, the 
operator in \eqref{3.90} is uniformly bounded in norm, proving \eqref{bd}. In fact, 
using $\bsB_n = \chi_n(\bsA_-) \bsB \chi_n(\bsA_-)$ again, and repeatedly employing 
commutativity of $\chi_n(\bsA_-)$ and $(\bsH_0 -z \, \bsI)^{-1}$, \eqref{3.90} also proves 
the strong convergence of $ \big(\bsA_- \bsB_n + \bsB_n \bsA_- + \bsB_n^2 - \bsB_n'\big)
(\bsH_0 - z \, \bsI)^{-1}$ to $ \big(\bsA_- \bsB + \bsB \bsA_- + \bsB^2 - \bsB'\big)
(\bsH_0 - z \, \bsI)^{-1}$ as $n \to \infty$ for $0 < |\Im(z)|$ sufficiently large. Using 
the fact that strong convergence of a sequence of uniformly bounded operators is 
equivalent to strong resolvent convergence of  the sequence, one obtains the asserted 
strong convergence in \eqref{conv2} for $0 < |\Im(z)|$ sufficiently large. An application of \eqref{limR} together with the bound \eqref{bd} permits one to extend this to all $z \in \bbC \backslash [0,\infty)$, completing the proof of \eqref{conv2}. 

Finally, to prove \eqref{conv1} it suffices to combine the strong resolvent convergence in \eqref{limR}, the uniform boundedness in \eqref{bd} with equality \eqref{closure}, the strong convergence   
\begin{equation} 
\slim_{n \to \infty} (\bsH_{1,n}-z \, \bsI)^{-1}(\bsH_{0} - z \, \bsI) f 
= ( \bsH_1 - z \, \bsI)^{-1}(\bsH_{0}-z \, \bsI)f, \quad f \in \dom(\bsH_0), 
\end{equation} 
and the fact that $\dom(\bsH_0)$ is dense in $L^2(\bbR; \cH)$. Here we 
used that  uniformly bounded sequences of bounded operators in a Hilbert space converge strongly if they converge pointwise on a dense subset of the Hilbert 
space. 
\end{proof}

Next, we recall that 
\begin{align}
& \bsH_2 - \bsH_1 = 2 \bsB^{\prime},      \lb{2.64} \\
&  \bsH_{2,n} - \bsH_{1,n} = 2 \bsB_n^{\prime} = 2 \chi_n(\bsA_-) \bsB^{\prime} \chi_n(\bsA_-), 
\quad n \in \bbN,    \lb{2.65} 
\end{align}
and in analogy to \eqref{3.31}--\eqref{3.32} one concludes that 
\begin{equation}
\big[(\bsH_{2,n}-z \, \bsI)^{-1} - (\bsH_{1,n}-z \, \bsI)^{-1}\big] \in \cB_1\big(L^2(\bbR; \cH)\big), 
\quad n \in \bbN, \; z \in \bbC \backslash [0,\infty),    \lb{2.65a}
\end{equation}
since 
\begin{align}
& (\bsH_{2,n}-z \, \bsI)^{-1} - (\bsH_{1,n}-z \, \bsI)^{-1} 
= - \big[\ol{(\bsH_{1,n}-z \, \bsI)^{-1}(\bsH_0 - z \, \bsI)}\big] \chi_n(\bsA_-)   \no \\
& \quad \times \big[(\bsH_0 -z \, \bsI)^{-1} 2 \bsB^{\prime} (\bsH_0 -z \, \bsI)^{-1} \big] 
\chi_n(\bsA_-)  \big[(\bsH_0 -z \, \bsI)(\bsH_{2,n} - z \, \bsI)^{-1}\big],     \lb{2.65b} \\
& \hspace*{8.2cm} n \in \bbN, \; z \in \bbC \backslash [0,\infty),    \no  
\end{align}
again employing commutativity of $\chi_n(\bsA_-)$ and $(\bsH_0 -z \, \bsI)^{-1}$ 
(cf.\ \eqref{comm}). 

Finally, we proceed to some crucial convergence results to be used in Section \ref{s5}. 

\begin{theorem} \lb{t3.7}
Assume Hypothesis \ref{h3.4}. Then 
\begin{align}
\begin{split} 
& \lim_{n\to\infty} \big\|\big[(\bsH_{2,n} - z \, \bsI)^{-1} - (\bsH_{1,n} - z \, \bsI)^{-1}\big]  \\
& \hspace*{1cm} - [(\bsH_2 - z \, \bsI)^{-1} - (\bsH_1 - z \, \bsI)^{-1}\big]
\big\|_{\cB_1(L^2(\bbR; \cH))} = 0, \quad z \in \bbC \backslash \bbR.    \lb{2.66}
\end{split} 
\end{align}
\end{theorem}
\begin{proof}
Equations \eqref{3.31} and \eqref{2.65} yield  
\begin{align}
& \big[(\bsH_{2,n} - z \, \bsI)^{-1} - (\bsH_{1,n} - z \, \bsI)^{-1}\big]
 - [(\bsH_2 - z \, \bsI)^{-1} - (\bsH_1 - z \, \bsI)^{-1}\big]    \no \\
 & \quad = - 2 (\bsH_{2,n} - z \, \bsI)^{-1} \bsB_n^{\prime} (\bsH_{1,n} - z \, \bsI)^{-1} 
 + 2 (\bsH_2 - z \, \bsI)^{-1} \bsB^{\prime} (\bsH_1 - z \, \bsI)^{-1}   \no \\
 & \quad = - 2\ol{\big[(\bsH_{2,n} - z \, \bsI)^{-1} (\bsH_0 - z \, \bsI)\big]}    \no \\
& \qquad \quad \times  \big\{\chi_n(\bsA_-) (\bsH_0 - z \, \bsI)^{-1} 
\bsB^{\prime} (\bsH_0 - z \, \bsI)^{-1} 
 \chi_n(\bsA_-)\big\}    \no \\
& \qquad \quad \times \big[(\bsH_0 - z \, \bsI) (\bsH_{1,n} - z \, \bsI)^{-1}\big],     \\
& \qquad + 2\ol{\big[(\bsH_2 - z \, \bsI)^{-1} (\bsH_0 - z \, \bsI)\big]}    \no \\
& \qquad \quad \times  \big\{(\bsH_0 - z \, \bsI)^{-1} \bsB^{\prime} 
(\bsH_0 - z \, \bsI)^{-1}\big\}    \no \\
& \qquad \quad \times \big[(\bsH_0 - z \, \bsI) (\bsH_1 - z \, \bsI)^{-1}\big], \quad 
z \in \bbC \backslash [0,\infty).  
\end{align}
By Lemma \ref{l3.6} and \eqref{3.25a}, the term $\big\{\chi_n(\bsA_-) (\bsH_0 - z \, \bsI)^{-1} 
\bsB^{\prime} (\bsH_0 - z \, \bsI)^{-1} 
 \chi_n(\bsA_-)\big\}$ converges to $\big\{(\bsH_0 - z \, \bsI)^{-1} \bsB^{\prime} 
(\bsH_0 - z \, \bsI)^{-1}\big\}$ in $\cB_1\big(L^2(\bbR; \cH)\big)$-norm as $n\to \infty$.
Another application of Lemma \ref{l3.6} proves \eqref{2.66} since by Lemma \ref{l3.5}\,$(ii)$, for $z \in \bbC \backslash [0,\infty)$, one has  
\begin{align}
& \slim_{n\to\infty} \big[(\bsH_0 - z \, \bsI) (\bsH_{1,n} - z \, \bsI)^{-1}\big] = 
\big[(\bsH_0 - z \, \bsI) (\bsH_1 - z \, \bsI)^{-1}\big],    \lb{2.69} \\
& \slim_{n\to\infty} \ol{\big[(\bsH_{2,n} - z \, \bsI)^{-1} (\bsH_0 - z \, \bsI)\big]} =  
\ol{\big[(\bsH_2 - z \, \bsI)^{-1} (\bsH_0 - z \, \bsI)\big]}.     \lb{2.70}
\end{align} 
\end{proof}

\begin{theorem} \lb{t3.8}
Assume Hypothesis \ref{h3.4} and let $z, z' \in \bbC \backslash [0,\infty)$. Then 
\begin{align}
& \lim_{n\to\infty} \big\|\bsB_n^{\prime} (\bsH_{j,n} - z \, \bsI)^{-1} 
- \bsB^{\prime} (\bsH_j - z \, \bsI)^{-1}
\big\|_{\cB_2(L^2(\bbR; \cH))} = 0, \quad j = 1,2,   \lb{2.71} \\
& \lim_{n\to\infty} \big\|(\bsH_{j,n} - z \, \bsI)^{-1} 2 \bsB_n^{\prime} (\bsH_{j,n} - z' \, \bsI)^{-1} 
\no \\
& \hspace*{1cm} - (\bsH_j - z \, \bsI)^{-1} 2 \bsB^{\prime} (\bsH_j - z' \, \bsI)^{-1}
\big\|_{\cB_1(L^2(\bbR; \cH))} = 0,  \quad j = 1,2.    \lb{2.72}
\end{align}
\end{theorem}
\begin{proof}
Fix $z, z' \in \bbC \backslash [0,\infty)$. To prove \eqref{2.71} one writes
\begin{align}
\begin{split} 
& \bsB_n^{\prime} (\bsH_{j,n} - z \, \bsI)^{-1} = \chi_n(\bsA_-) \big[\bsB^{\prime}  
 (\bsH_0 - z \, \bsI)^{-1}\big]    \\
& \quad \times \chi_n(\bsA_-) \big[(\bsH_0 - z \, \bsI) (\bsH_{j,n} - z \, \bsI)^{-1}\big], 
\quad j =1,2, \; n \in \bbN, 
\end{split} 
\end{align}
employing once again commutativity of $\chi_n(\bsA_-)$ and $(\bsH_0 - z \, \bsI)^{-1}$ 
(cf.\ \eqref{comm}). Since 
\begin{equation} 
\bsB^{\prime} (\bsH_0 - z \, \bsI)^{-1} 
= \big[U_{\bsB^{\prime}}  |\bsB^{\prime}|^{1/2}\big] |\bsB^{\prime}|^{1/2} (\bsH_0 - z \, \bsI)^{-1}  
\in \cB_2\big(L^2(\bbR; \cH)\big),   \lb{2.74}
\end{equation} 
by \eqref{3.25} and the polar decomposition 
$\bsB^{\prime} = U_{\bsB^{\prime}} |\bsB^{\prime}|$ of $\bsB^{\prime} \in \cB\big(L^2(\bbR; \cH)\big)$  
(cf.\ \eqref{B'bd}). Thus, \eqref{2.71} is a consequence of Lemma \ref{l3.6} combined with 
\eqref{2.75} and \eqref{conv2}.

Relation \eqref{2.72} follows along exactly the same lines upon decomposing 
\begin{align}
& (\bsH_{j,n} - z \, \bsI)^{-1} 2 \bsB_n^{\prime} (\bsH_{j,n} - z' \, \bsI)^{-1} = 
\big[\ol{(\bsH_{j,n} - z \, \bsI)^{-1} (\bsH_0 - z \, \bsI)}\big] \chi_n(\bsA_-)      \no \\
& \quad \times \big[(\bsH_0 - z \, \bsI)^{-1} 2 \bsB^{\prime} (\bsH_0 - z' \, \bsI)^{-1}\big] 
\chi_n(\bsA_-) \big[(\bsH_0 - z' \, \bsI) (\bsH_{j,n} - z' \, \bsI)^{-1}\big],     \no \\
& \hspace*{8.6cm}  j =1,2, \; n \in \bbN,     
\end{align} 
applying once more Lemma \ref{l3.6}, \eqref{conv1}, \eqref{conv2}, \eqref{2.75}, \eqref{comm},  
and 
\begin{align} 
\begin{split}  
& (\bsH_0 - z \, \bsI)^{-1} \bsB^{\prime} (\bsH_0 - z' \, \bsI)^{-1} 
= \big[(\bsH_0 - z \, \bsI)^{-1} |\bsB^{\prime}|^{1/2}\big] \sgn(\bsB^{\prime})    \\
& \quad \times \big[ |\bsB^{\prime}|^{1/2} (\bsH_0 - z' \, \bsI)^{-1}\big] 
\in \cB_1\big(L^2(\bbR; \cH)\big),     \lb{2.77}
\end{split} 
\end{align} 
since $|\bsB^{\prime}|^{1/2} (\bsH_0 - z \, \bsI)^{-1} \in \cB_2\big(L^2(\bbR; \cH)\big)$ 
by \eqref{3.25}. 
\end{proof}

\section{Computing $\xi(\, \cdot \, ; \bsH_2, \bsH_1)$  In Terms Of 
$\xi(\, \cdot \,; A_+, A_-)$} \lb{s5}

Given the results of Section \ref{s3} and Appendix \ref{sA}, we now determine 
$\xi(\, \cdot \, ; \bsH_2, \bsH_1)$ in terms of $\xi(\, \cdot \,; A_+, A_-)$. This represents 
one of the principal results of this paper.

\begin{theorem}\lb{t4.4}
Assume Hypothesis \ref{h3.4}. Then, 
\begin{align} 
\begin{split} 
& \int_{[0,\infty)} \xi(\lambda; \bsH_2, \bsH_1) \, d \lambda 
\, \big[(\lambda - z)^{-1} - (\lambda - z_0)^{-1}\big]     \\ 
& \quad = \int_{\bbR} \xi(\nu; A_+, A_-) \, d\nu \big[(\nu^2 - z)^{-1/2} - (\nu^2 - z_0)^{-1/2}\big], 
\quad z, z_0 \in \bbC \backslash [0,\infty).    \lb{5.1} 
\end{split} 
\end{align}
Moreover, 
\begin{equation}
\xi(\lambda; \bsH_2, \bsH_1) = \f{1}{\pi} \int_{- \lambda^{1/2}}^{\lambda^{1/2}} 
\f{\xi(\nu; A_+, A_-) \, d \nu}{(\lambda - \nu^2)^{1/2}} \, \text{ for a.e.~$\lambda > 0$.} 
\lb{5.2} 
\end{equation}
\end{theorem}
\begin{proof} 
Due to relation \eqref{intB'n}, \cite{CGPST14} and \cite{Pu08} apply and one concludes the approximate trace formula,
\begin{align}
\begin{split} 
& \tr_{L^2(\bbR; \cH)}\big((\bsH_{2,n} - z \, \bsI)^{-1}-(\bsH_{1,n} - z \, 
\bsI)^{-1}\big)    \\
& \quad = \f{1}{2z} \tr_{\cH} \big(g_z(A_{+,n})-g_z(A_-)\big),   \quad 
n \in \bbN, \; z\in \bbC \backslash [0,\infty),     \lb{trn}
\end{split} 
\end{align} 
with
\begin{equation}   
g_z(x) = x(x^2-z)^{-1/2}, \quad z\in\C\backslash [0,\infty), \; x\in\bbR.   
\end{equation} 
Relation \eqref{trn}  and the Krein--Lifshitz trace formula yield 
\begin{equation}
\int_{[0,\infty)} \f{\xi(\lambda; \bsH_{2,n}, \bsH_{1,n}) \, d\lambda}{(\lambda - z)^2} 
= \f{1}{2} \int_{\bbR} \f{\xi(\nu; A_{+,n}, A_-) \, d\nu}{(\nu^2 -z)^{3/2}}, \quad 
n \in \bbN, \; z \in \bbC \backslash [0,\infty).      \lb{krn}
\end{equation}
As shown in the course of the proof of  Theorem\ 8.2 in \cite{GLMST11}, \eqref{krn} 
implies the relation
\begin{align} 
& \int_{[0,\infty)} \xi(\lambda; \bsH_{2,n}, \bsH_{1,n}) \, d \lambda 
\big[(\lambda - z)^{-1} - (\lambda - z_0)^{-1}\big]     \no \\
& \quad = \int_{\bbR} \xi(\nu; A_{+,n}, A_-) \, d\nu \big[(\nu^2 - z)^{-1/2} - (\nu^2 - z_0)^{-1/2}\big],    \lb{krtrn} \\
& \hspace*{5cm}  n \in \bbN, \; z, z_0 \in \bbC \backslash [0,\infty).   \no
\end{align}

Combining Theorem \ref{t3.7} with the Krein--Lifshitz trace formula 
\eqref{3.45} (for the pair $(\bsH_2,\bsH_1)$ as well as the pairs 
$(\bsH_{2,n}, \bsH_{1,n})$, $n\in\bbN$), yields
\begin{align}
& \lim_{n\to\infty} \int_{[0,\infty)} \f{\xi(\lambda; \bsH_{2,n}, \bsH_{1,n}) 
\, d\lambda}{(\lambda - z)^2}    \no \\
& \quad = - \lim_{n\to\infty} \tr_{L^2(\bbR; \cH)}\big((\bsH_{2,n} - z \, \bsI)^{-1} 
- (\bsH_{1,n} - z \, \bsI)^{-1}\big)    \no \\
& \quad = - \tr_{L^2(\bbR; \cH)}\big((\bsH_2 - z \, \bsI)^{-1} - (\bsH_1 - z \, \bsI)^{-1}\big)  \no \\
& \quad = \int_{[0,\infty)} \f{\xi(\lambda; \bsH_2, \bsH_1) 
\, d\lambda}{(\lambda - z)^2}, \quad z \in \bbC \backslash \bbR.     \lb{trlimH}
\end{align}
Lemma \ref{l3.5}\,$(i)$ and Theorem \ref{t3.7} imply that the pairs of self-adjoint operators 
$(\bsH_{2,n}, \bsH_{1,n})$, $n \in \bbN$, and $(\bsH_2,\bsH_1)$ satisfy the hypotheses 
\eqref{detcont}, \eqref{det},\eqref{B2conv}, \eqref{B1conve} (identifying the pairs $(A_n, A_{0,n})$ and $(A,A_0)$ with the pairs $(\bsH_{2,n}, \bsH_{1,n})$ and $(\bsH_2,\bsH_1)$, respectively). Thus, an 
application of \eqref{detlim} to the pairs $(\bsH_{2,n}, \bsH_{1,n})$ and $(\bsH_2,\bsH_1)$ 
implies 
\begin{align}
\begin{split} 
& \lim_{n\to\infty} \int_{[0,\infty)} \xi(\lambda; \bsH_{2,n}, \bsH_{1,n}) 
\, d\lambda \big[(\lambda - z)^{-1} - (\lambda - z_0)^{-1}\big]     \\
& \quad = \int_{[0,\infty)} \xi(\lambda; \bsH_2, \bsH_1) 
\, d\lambda\big[(\lambda - z)^{-1} - (\lambda - z_0)^{-1}\big] , 
\quad z, z_0 \in \bbC \backslash \bbR.     \lb{trH} 
\end{split} 
\end{align}
On the other hand, since 
\begin{equation}
\big[(\nu^2 - z)^{-1/2} - (\nu^2 - z_0)^{-1/2}\big] \underset{|\nu| \to \infty}{=} 
\Oh\big(|\nu|^{-3}\big), \quad z, z_0 \in \bbC \backslash \bbR, 
\end{equation}
and hence $(\nu^2 + 1) \big[(\nu^2 - z)^{-1/2} - (\nu^2 - z_0)^{-1/2}\big]$ is uniformly bounded 
for $\nu \in \bbR$, \eqref{B.119} yields 
\begin{align} 
\begin{split}
& \lim_{n \to \infty} \int_{\bbR} \xi(\nu; A_{+,n}, A_-) d\nu \, 
\big[(\nu^2 - z)^{-1/2} - (\nu^2 - z_0)^{-1/2}\big]    \\
& \quad = \int_{\bbR} \xi(\nu; A_+, A_-) d\nu \, \big[(\nu^2 - z)^{-1/2} - (\nu^2 - z_0)^{-1/2}\big], 
\quad z, z_0 \in \bbC \backslash [0,\infty).    \lb{xiA} 
\end{split}
\end{align}
Thus, combining \eqref{krtrn}, \eqref{trH}, and \eqref{xiA} one finally obtains 
\begin{align} 
& \int_{[0,\infty)} \xi(\lambda; \bsH_2, \bsH_1) \, d \lambda 
\, \big[(\lambda - z)^{-1} - (\lambda - z_0)^{-1}\big]    \no \\
& \quad = \lim_{n \to \infty} \int_{[0,\infty)} \xi(\lambda; \bsH_{2,n}, \bsH_{1,n}) \, d \lambda 
\, \big[(\lambda - z)^{-1} - (\lambda - z_0)^{-1}\big]    \no \\
& \quad = \lim_{n \to \infty} \int_{\bbR} \xi(\nu; A_{+,n}, A_-) \, d\nu 
\big[(\nu^2 - z)^{-1/2} - (\nu^2 - z_0)^{-1/2}\big]    \no \\
& \quad = \int_{\bbR} \xi(\nu; A_+, A_-) \, d\nu \big[(\nu^2 - z)^{-1/2} - (\nu^2 - z_0)^{-1/2}\big], 
\quad z, z_0 \in \bbC \backslash [0,\infty),    \lb{xixi} 
\end{align}
and hence \eqref{5.1}.

Applying the Stieltjes inversion formula (see the discussion in \cite{AD56} and 
in \cite[Appendix~B]{We80}) to \eqref{xixi} then proves \eqref{5.2} precisely along 
the lines detailed in the proof of \cite[Theorem~8.2]{GLMST11}.  
\end{proof}

Equation \eqref{5.2} now represents a far reaching extension of 
Pushnitski's formula originally obtained in \cite{Pu08} and considerably generalized in 
\cite{GLMST11}. In particular, the relative trace class assumption employed 
in \cite{GLMST11} has now been removed.

\section{The Witten Index} \lb{s6}

In this section we briefly discuss the notion of the Witten index for $\bsD_\bsA^{}$ 
following the detailed treatment in \cite{CGPST14a}. The results of the present paper 
now enable us to remove the ``relatively trace class perturbation assumption'' in \cite{CGPST14a} as well 
as the Fredholm hypothesis in \cite{GLMST11}. 

\begin{definition} \lb{d8.1} 
Let $T$ be a closed, linear, densely defined operator in $\cH$ and  
suppose that for some $($and hence for all\,$)$ 
$z \in \bbC \backslash [0,\infty)$,  
\begin{equation} 
\big[(T^* T - z I_{\cH})^{-1} - (TT^* - z I_{\cH})^{-1}\big] \in \cB_1(\cH).   \lb{8.1} 
\end{equation}  
Then introducing the resolvent regularization 
\begin{equation}
\Delta_r(T, \lambda) = (- \lambda) \tr_{\cH}\big((T^* T - \lambda I_{\cH})^{-1}
- (T T^* - \lambda I_{\cH})^{-1}\big), \quad \lambda < 0,        \lb{8.2} 
\end{equation} 
the resolvent regularized Witten index $W_r (T)$ of $T$ is defined by  
\begin{equation} 
W_r(T) = \lim_{\lambda \uparrow 0} \Delta_r(T, \lambda),      \lb{8.3}
\end{equation}
whenever this limit exists. 
\end{definition} 

Here, in obvious notation, the subscript ``$r$'' indicates the use of the resolvent 
regularization (for a semigroup or heat kernel regularization we refer to \cite{CGPST14a}; 
the heat kernel regularization yields results consistent with the resolvent regularization, see, 
\cite{CGPST14a}). Before proceeding to compute the Witten index we recall the known 
consistency between the Fredholm and Witten index whenever $T$ is Fredholm:

\begin{theorem} $($\cite{BGGSS87}, \cite{GS88}.$)$  \lb{t8.2} 
Suppose that $T$ is a Fredholm operator in $\cH$. If \eqref{8.1} holds, then the 
resolvent regularized Witten index $W_r(T)$ exists, equals the Fredholm index, 
$\ind (T)$, of $T$, and
\begin{equation} 
W_r(T) =  \ind (T) = \xi(0_+; T T^*, T^* T).      \lb{8.4}
\end{equation}
\end{theorem}

The following result is proved in \cite[Theorem~2.6]{CGPST14a} under the 
assumption of a relatively trace class perturbation, however, the argument can  
be adapted to the present setting.

\begin{theorem}\label{t8.iff}
Assume Hypothesis \ref{h3.1}. Then the operator $\bsD_\bsA^{}$ is Fredholm if and only if  
$0 \in \rho(A_+) \cap \rho(A_-)$.
\end{theorem}

Since generally, $\bsD_\bsA^{}$ is not a Fredholm operator in $L^2(\bbR; \cH)$, we now determine 
the resolvent regularized Witten index of $\bsD_\bsA^{}$ as follows:

\begin{theorem} \lb{t8.3} 
Assume Hypothesis \ref{h3.4} and assume that $0$ is a right 
and a left Lebesgue point of $\xi(\,\cdot\,\, ; A_+, A_-)$ $($denoted by $\Lxi(0_+; A_+,A_-)$ 
and $\Lxi(0_-; A_+, A_-)$$)$. Then $0$ is a right Lebesgue point of 
$\xi(\,\cdot\,\, ; \bsH_2, \bsH_1)$ $($denoted by $\Lxi(0_+; \bsH_2, \bsH_1)$$)$ 
and $W_r(\bsD_\bsA^{})$ exists and equals 
\begin{equation}
W_r(\bsD_\bsA^{}) = \Lxi(0_+; \bsH_2, \bsH_1) 
= [\Lxi(0_+; A_+,A_-) + \Lxi(0_-; A_+, A_-)]/2.     \lb{8.5}
\end{equation}
In particular, if $0 \in \rho(A_+) \cap \rho(A_-)$, then $\bsD_\bsA^{}$ is Fredholm and
\begin{equation} 
\ind(\bsD_\bsA^{}) = W_r(\bsD_\bsA^{}) = \xi(0; A_+, A_-).     \lb{8.6} 
\end{equation} 
\end{theorem}
\begin{proof}
One can closely follow the argument as used in  \cite{CGPST14a}.
First, one rewrites \eqref{5.2} in the form,
\begin{equation} 
\xi(\lambda; \bsH_2, \bsH_1) = \frac{1}{\pi}\int_0^{\lambda^{1/2}}
\frac{d \nu \, [\xi(\nu; A_+,A_-) + \xi(-\nu; A_+,A_-)]}{(\lambda-\nu^2)^{1/2}},  
\quad \lambda > 0.     \lb{55f}
\end{equation} 
Applying \cite[Lemma~4.1\,(i)]{CGPST14a} to the function 
$f(\nu) = [\xi(\nu,A_+,A_-)+\xi(-\nu,A_+,A_-)]$, 
$\nu > 0$, yields the Lebesgue point statement for $\xi(\,\cdot\,\, ; \bsH_2, \bsH_1)$ and 
also proves that $\Lxi(0_+; \bsH_2, \bsH_1) = [\Lxi(0_+; A_+,A_-) + \Lxi(0_-; A_+, A_-)]/2$. 

Next, we note that combining Corollary \ref{cB.9}, \eqref{krn}, \eqref{trlimH} yields
\begin{align}
& - z\tr_{L^2(\bbR;\cH)} \big((\bsH_2 - z \bsI)^{-1} - (\bsH_1 - z \bsI)^{-1}\big) 
= z \int_{[0,\infty)} \f{\xi(\lambda'; \bsH_2, \bsH_1)\, d\lambda'}{(\lambda' - z)^2}   \no \\
& \quad = z \lim_{n \to \infty} \int_{[0,\infty)} \f{\xi(\lambda'; \bsH_{2,n}, \bsH_{1,n})
\, d\lambda'}{(\lambda' - z)^2}   
= \f{z}{2} \lim_{n \to \infty} \int_{\bbR} \frac{\xi(\nu; A_{+,n}, A_-) \, d\nu}{(\nu^2 - z)^{3/2}}, 
\no \\
& \quad = \f{z}{2} \int_{\bbR} \frac{\xi(\nu; A_+, A_-) \, d\nu}{(\nu^2 - z)^{3/2}} 
\quad z \in \bbC \backslash [0,\infty).      \lb{8.7} 
\end{align}
Thus, one obtains existence of $W_r(\bsD_\bsA^{})$ (cf.\ \eqref{8.3}) and its equality with 
the expression $[\Lxi(0_+; A_+,A_-) + \Lxi(0_-; A_+, A_-)]/2$ upon applying 
\cite[Lemma~4.2]{CGPST14a} to the last term in \eqref{8.7}, with $z = \lambda < 0$. 

In the case where $A_\pm$ are boundedly invertible, the equality 
$\ind(\bsD_\bsA^{}) = W_r(\bsD_\bsA^{})$ immediately follows from Theorems \ref{t8.2} and \ref{t8.iff}. In addition, since $0\in\rho(A_+)\cap\rho(A_-)$, the function $\xi(\cdot; A_+, A_-)$ 
is constant on some interval $(-\varepsilon,\varepsilon)$ for some $\varepsilon>0$ (see, e.g., 
\cite[p.~300]{Ya92}), and hence $\Lxi(0_+; A_+,A_-)=\Lxi(0_-; A_+, A_-)=\xi(0; A_+,A_-)$ 
yields \eqref{8.6}. 
\end{proof}

In general (i.e., if $T$ is not Fredholm), $W_r(T)$ is not necessarily 
integer-valued; in fact, it can take on any prescribed real number (cf., e.g., the 
analysis in \cite{An90a}, \cite{BGGSS87}). In this context we recall the crucial fact 
that $W_r(T)$ has stability properties with respect to additive perturbations 
analogous to the Fredholm index, as long as one replaces the familiar 
relative compactness assumption on the perturbation in connection with the 
Fredholm index, by appropriate relative trace class conditions in connection with 
the Witten index, as shown in \cite{BGGSS87} and \cite{GS88}. In this context we also 
refer to \cite{CP74}.

\section{A $(1+1)$-Dimensional Example} \lb{s8}

In our final section we briefly illustrate our formalism in terms of a   
concrete $(1+1)$-dimensional example treated in great detail in \cite{CGLPSZ14} and 
\cite{CGLPSZ14a}. 

\begin{hypothesis} \lb{h9.1} 
Suppose the real-valued functions $\phi, \theta$ satisfy  
\begin{align}
& \, \phi \in AC_{\loc}(\bbR) \cap L^{\infty}(\bbR) \cap L^1(\bbR),  
\; \phi' \in L^{\infty}(\bbR),     \\
\begin{split} 
& \, \theta \in AC_{\loc}(\bbR) \cap L^{\infty}(\bbR), \; 
\theta' \in L^{\infty}(\bbR) \cap L^1(\bbR),      \\
& \lim_{t \to \infty} \theta (t) = 1, \; \lim_{t \to - \infty} \theta (t) = 0. 
\end{split} 
\end{align} 
\end{hypothesis}

Given Hypothesis \ref{h9.1}, we introduce the family of self-adjoint operators 
$A(t)$, $t \in \bbR$, in $L^2(\bbR)$, 
\begin{equation}
A(t) = - i \f{d}{dx} + \theta(t) \phi, \quad 
\dom(A(t)) = W^{1,2}(\bbR), \; t \in \bbR.     \lb{9.3} 
\end{equation}
Its  asymptotes as $t \to \pm \infty$ are given by 
\begin{align} 
& \nlim_{t \to \pm \infty} (A(t) - z I)^{-1} = (A_{\pm} - z I)^{-1}, \quad z \in \bbC \backslash \bbR, \\  
& A_+ = - i \f{d}{dx} + \phi, \quad A_- =  - i \f{d}{dx}, \quad 
\dom(A_{\pm}) = W^{1,2}(\bbR).  
\end{align}
(For simplicity, we adopt the abbreviation $I = I_{L^2(\bbR)}$ throughout this section.)

In addition, it is convenient to introduce the family of bounded operators 
$B(t)$, $t \in \bbR$, in $L^2(\bbR)$, where
\begin{equation}
B(t) = \theta(t) \phi, \quad \dom(B(t)) = L^{2}(\bbR), \;  t \in \bbR,
\end{equation}
implying 
\begin{equation} 
A(t) = A_- + B(t), \quad t \in \bbR.  
\end{equation} 
The asymptotes of $B(t)$, $t \in \bbR$, as $t \to \pm \infty$ are then given by 
\begin{equation} 
B_+ = \nlim_{t \to + \infty} B(t) = \phi, \quad B_- = \nlim_{t \to - \infty} B(t) = 0.  
\end{equation}

Introducing the operator $d/dt$ in $L^2\big(\bbR; dt; L^2(\bbR;dx)\big)$  by 
\begin{align}
& \bigg(\f{d}{dt}f\bigg)(t) = f'(t) \, \text{ for a.e.\ $t\in\bbR$,}    \no \\
& \, f \in \dom(d/dt) = \big\{g \in L^2\big(\bbR;dt;L^2(\bbR)\big) \, \big|\,
g \in AC_{\loc}\big(\bbR; L^2(\bbR)\big), \\
& \hspace*{6.3cm} g' \in L^2\big(\bbR;dt;L^2(\bbR)\big)\big\}   \no \\
& \hspace*{2.3cm} = W^{1,2} \big(\bbR; dt; L^2(\bbR; dx)\big).      \label{2.ddtR} 
\end{align} 
one now defines $\bsA$, $\bsB$, $\bsA_-$, $\bsA' = \bsB'$, $\bsD_\bsA^{}$, $\bsH_0$, 
and the pair $(\bsH_2, \bsH_1)$ in the Hilbert space $L^2\big(\bbR; dt; L^2(\bbR; dx)\big)$ 
as in Section \ref{s3} and for notational simplicity we agree to identify 
$L^2\big(\bbR; dt; L^2(\bbR; dx)\big)$ with $L^2(\bbR^2; dt dx)$ in the following. In particular,
$\bsD_\bsA^{}$ in $L^2(\bbR^2)$ is of the form  
\begin{equation}
\bsD_\bsA^{} = \f{d}{dt} + \bsA,
\quad \dom(\bsD_\bsA^{})= W^{1,2}(\bbR^2),    
\end{equation}
with $\bsA$ defined as in \eqref{1.1} identifying $\cH = L^2(\bbR)$ and $A(t)$, $t \in \bbR$, 
is given by \eqref{9.3}.

Similarly, mimicking the approximation setup described \eqref{3.chin}--\eqref{Antprime} one introduces $A_n(t)$, its asymptotes $A_{\pm, n}$ as 
$t \to \pm \infty$, $B_n(t)$, $t \in \bbR$, $n \in \bbN$, and then also defines 
$\bsA_n$, $\bsB_n$, $\bsA'_n = \bsB'_n$, and the pair $(\bsH_{2,n}, \bsH_{1,n})$ in 
$L^2(\bbR^2; dt dx)$ as in Section \ref{s3}. 

As shown in \cite{CGLPSZ14}, the assumptions on $\phi$ and $\theta$ made in 
Hypothesis \ref{h9.1} guarantee that all conditions in Hypothesis \ref{h3.4} are met and 
the following results can be derived: 
\begin{equation}
\big[(A_+ - z I)^{-1} - (A_- - z I)^{-1}\big] \in \cB_1\big(L^2(\bbR)\big), 
\quad z \in \bbC \backslash \bbR,     \lb{9.13} 
\end{equation}
and thus, the spectral shift function $\xi(\, \cdot \, ; A_+, A_-)$ for the pair $(A_+, A_-)$ exists 
and is well-defined up to an arbitrary additive real constant, satisfying 
\begin{equation}
\xi(\, \cdot \, ; A_+, A_-) \in L^1\big(\bbR; (\nu^2 + 1)^{-1} d\nu\big). 
\end{equation}

Introducing $\chi_n(A_-) = n (A_-^2 + n^2 I)^{-1/2}$, $n \in \bbN$, according to \eqref{3.chin}, 
the fact 
\begin{equation}
A_{+,n} - A_- = \chi_n(A_-) B_+ \chi_n(A_-) \in \cB_1\big(L^2(\bbR)\big), \quad n \in \bbN, 
\end{equation}
implies that the spectral shift functions 
$\xi(\, \cdot \, ; A_{+,n}, A_-)$, $n \in \bbN$, exist and are uniquely determined by 
\begin{equation}
\xi(\, \cdot \, ; A_{+,n}, A_-) \in L^1(\bbR; d\nu), \quad n \in \bbN. 
\end{equation}
In fact, one can derive the expressions 
\begin{align} 
& \xi(\nu; A_{+,n}, A_-)  = \pi^{-1} \Im\big(\ln\big({\det}_{2, L^2(\bbR)}
\big(I + \sgn(\phi) |\phi|^{1/2} \chi_n(A_-) (A_- - (\nu + i 0) I)^{-1}    \no \\ 
& \quad \times \chi_n(A_-) |\phi|^{1/2}
\big)\big)\big) + \f{1}{2 \pi} \f{n^2}{\nu^2 + n^2} \int_{\bbR} dx \, \phi(x)  \,     
\text{ for a.e.\ } \nu \in \bbR, \; n \in \bbN,      \lb{B.45} 
\end{align} 
and when studying the limit $n\to\infty$ of $\xi(\, \cdot \, ; A_{+,n}, A_-)$ one can prove that 
\begin{equation}
\lim_{n \to \infty} \xi(\nu; A_{+,n}, A_-) = \f{1}{2 \pi} \int_{\bbR} dx \, \phi(x), 
\quad \nu \in \bbR. 
\end{equation}
In addition, one can show that Theorem \ref{tB.8} applies and hence $\xi(\, \cdot \, ; A_+, A_-)$ 
associated with the pair $(A_+, A_-)$, normalized according to \eqref{B.109} (see also the 
discussion in Appendix \ref{sA}, particularly, \eqref{B.36a}) is determined via
\begin{equation}
\lim_{n \to \infty} \xi(\nu; A_{+,n}, A_-) = \f{1}{2 \pi} \int_{\bbR} dx \, \phi(x) 
= \xi(\nu; A_+, A_-), \quad \nu \in \bbR.
\end{equation}
Thus, one obtains the remarkable fact that $\xi(\, \cdot \, ; A_+, A_-)$ turns out to be 
constant in this example. (This phenomenon is explored and explained in detail in  
\cite{CGLPSZ14} in terms of scattering theoretic notions).

Similarly, the facts,
\begin{align}
& \big[(\bsH_2 - z \, \bsI)^{-1} - (\bsH_1 - z \, \bsI)^{-1}\big] \in 
\cB_1\big(L^2(\bbR^2)\big), \quad z \in \bbC \backslash [0,\infty),      \\
& \big[(\bsH_{2,n}-z \, \bsI)^{-1} - (\bsH_{1,n}-z \, \bsI)^{-1}\big] \in \cB_1\big(L^2(\bbR^2)\big), 
\quad n \in \bbN, \; z \in \bbC \backslash [0,\infty),   
\end{align} 
show that the spectral shift functions $\xi(\, \cdot \, ; \bsH_2, \bsH_1)$ and 
$\xi(\, \cdot \, ; \bsH_{2,n}, \bsH_{1,n})$ for the pairs $(\bsH_2, \bsH_1)$ and 
$(\bsH_2, \bsH_1)$, $n \in \bbN$, respectively, are well-defined. In particular, they satisfy 
\begin{equation}
\xi(\, \cdot \, ; \bsH_2, \bsH_1), \, \xi(\, \cdot \, ; \bsH_{2,n}, \bsH_{1,n})
 \in L^1\big(\bbR; (\lambda^2 + 1)^{-1} d\lambda\big), \quad n \in \bbN, 
\end{equation} 
and since $\bsH_j\geq 0$, $\bsH_{j,n} \geq 0$, $n \in \bbN$, $j=1,2$, one uniquely introduces 
$\xi(\,\cdot\,; \bsH_2,\bsH_1)$ and $\xi(\, \cdot \, ; \bsH_{2,n}, \bsH_{1,n})$, $n\in\bbN$, 
by requiring that
\begin{equation}
\xi(\lambda; \bsH_2,\bsH_1) = 0, \quad  
\xi(\, \cdot \, ; \bsH_{2,n}, \bsH_{1,n}) = 0, \quad \lambda < 0, \; n \in \bbN.   
\end{equation}

As shown in \cite{CGLPSZ14}, one can now prove the following intimate connection between 
$\xi(\, \cdot \,; A_{+,n}, A_-)$ and $\xi(\, \cdot \, ; \bsH_{2,n}, \bsH_{1,n})$, $n \in \bbN$, 
the Pushnitski-type formula, 
\begin{equation}
\xi(\lambda; \bsH_{2,n}, \bsH_{1,n}) = \f{1}{\pi} \int_{- \lambda^{1/2}}^{\lambda^{1/2}} 
\f{\xi(\nu; A_{+,n}, A_-) d \nu}{(\lambda - \nu^2)^{1/2}} \, \text{ for a.e.~$\lambda > 0$, 
$n\in\bbN$.}    \lb{9.22}
\end{equation}
A careful investigation in \cite{CGLPSZ14} establishes the analog of \eqref{9.22} in the 
limit $n \to \infty$. However, we emphasize the following formula is not derived in 
\cite{CGLPSZ14} by attempting to take the limit $n \to \infty$ of either side in \eqref{9.22}; 
instead it is derived via a careful application of various trace formuas and the Stieltjes 
inversion formula resulting in 
\begin{equation}
\xi(\lambda; \bsH_2, \bsH_1) = \xi (\nu; A_+, A_-) = \f{1}{2 \pi} \int_{\bbR} dx \, \phi(x)    \lb{9.23}
\end{equation}
for a.e.~$\lambda > 0$ and a.e.~$\nu \in \bbR$. As a consequence of \eqref{9.23}, the 
Witten index $W_r(\bsD_\bsA^{})$ of the non-Fredholm operator $\bsD_\bsA^{}$ exists 
and equals 
\begin{equation}
W_r(\bsD_\bsA^{}) = \xi(0_+; \bsH_2, \bsH_1) = 
\xi(0; A_+, A_-) = \f{1}{2 \pi} \int_{\bbR} dx \, \phi(x).     \lb{9.24}
\end{equation}

\appendix
\section{Some Facts On Spectral Shift Functions, Trace Formulas, and Modified 
Fredholm Determinants} \lb{sA}
\renewcommand{\theequation}{A.\arabic{equation}}
\renewcommand{\thetheorem}{A.\arabic{theorem}}
\setcounter{theorem}{0} \setcounter{equation}{0}

We recall a few basic facts on spectral shift functions employed in the 
bulk of this paper and provide results on trace formulas in terms of modified Fredholm 
determinants.
 
Closely following \cite[Sects.~2--6]{BY93} and \cite[Ch.~8]{Ya92}, we provide a brief 
discussion of how to restrict the open constant in the definition of the spectral shift function 
$\xi(\, \cdot \, ; A, A_0)$ up to an integer for a pair of self-adjoint operators $(A,A_0)$ in $\cH$ 
satisfying for some (and hence for all) $z_0 \in \rho(A) \cap \rho(A_0)$, 
\begin{equation}
\big[(A - z_0 I_{\cH})^{-1} - (A_0 - z_0 I_{\cH})^{-1}\big] \in \cB_1(\cH).    \lb{B.26a} 
\end{equation}
Motivated by the unitary Cayley transforms of $A$ and $A_0$, one introduces the modified 
perturbation determinant,
\begin{align}
\begin{split} 
\wti D_{A/A_0}(z;z_0) = {\det}_{\cH} 
\big((A - z I_{\cH})(A - \ol{z_0} I_{\cH})^{-1} (A_0 - \ol{z_0} I_{\cH})(A_0 - z I_{\cH})^{-1}\big),&  \\     
z \in \rho(A) \cap \rho(A_0), \; \Im(z_0) > 0,&     \lb{B.27a}
\end{split} 
\end{align}
and notes that (cf.\ \cite[p.270]{Ya92}) 
\begin{equation}
\ol{\wti D_{A/A_0} (z; z_0)} = \wti D_{A/A_0} (\ol z; z_0)/ \wti D_{A/A_0} (z_0; z_0), \quad 
 \wti D_{A/A_0} (\ol{z_0}; z_0) =1, 
\end{equation}
and 
\begin{align}
\begin{split} 
{\tr}_{\cH} \big[(A - z I_{\cH})^{-1} - (A_0 - z I_{\cH})^{-1}\big] 
= - \f{d}{dz} \ln\big( \wti D_{A/A_0} (z; z_0)\big),&  \\ 
z \in \rho(A) \cap \rho(A_0), \; \Im(z_0) > 0.& 
\end{split}
\end{align}
In addition,
\begin{equation}
\f{\wti D_{A/A_0} (z; z_0)}{\wti D_{A/A_0} (\ol{z}; z_0)} 
= \f{\wti D_{A/A_0} (z; z_1)}{\wti D_{A/A_0} (\ol{z}; z_1)}, \quad   
z \in \rho(A) \cap \rho(A_0), \; \Im(z_0) > 0, \, \Im(z_1) > 0. 
\end{equation}
Then, defining
\begin{align}
\begin{split}
& \xi(\lambda; A,A_0; z_0) = (2\pi)^{-1} \lim_{\varepsilon \downarrow 0} 
\big[\Im\big(\ln\big(\wti D_{A/A_0} (\lambda + i \varepsilon; z_0)\big)\big)  \lb{B.31a} \\
& \hspace*{4.2cm} - \Im\big(\ln\big(\wti D_{A/A_0} (\lambda - i \varepsilon; z_0)\big)\big)\big]
\, \text{ for a.e.~$\lambda \in \bbR$,} 
\end{split}
\end{align}
one obtains for $z \in \rho(A) \cap \rho(A_0)$, $\Im(z_0) > 0$, $\Im(z_1) > 0$, 
\begin{align}
&  \xi(\, \cdot \,; A,A_0; z_0) \in L^1\big(\bbR; (\lambda^2 + 1)^{-1} d\lambda\big),    \\
& \ln\big(\wti D_{A/A_0} (z; z_0)\big) = \int_{\bbR} \xi(\lambda; A,A_0; z_0) d\lambda 
\big[(\lambda -z)^{-1} - (\lambda - \ol{z_0})^{-1}\big],    \\
& \xi(\lambda; A,A_0; z_0) = \xi(\lambda; A,A_0; z_1) + n(z_0,z_1) \, 
\text{ for some $n(z_0,z_1) \in \bbZ$,}     \lb{B.34a} \\
& {\tr}_{\cH} \big[(A - z)^{-1} - (A_0 - z I_{\cH})^{-1}\big] 
= - \int_{\bbR} \f{\xi(\lambda; A,A_0; z_0) d \lambda}{(\lambda - z)^2},    \\
& [f(A) - f(A_0)] \in \cB_1(\cH), \quad f \in C_0^{\infty}(\bbR),    \\
& {\tr}_{\cH} (f(A) - f(A_0)) = 
\int_{\bbR} \xi(\lambda; A,A_0; z_0) d\lambda \, f'(\lambda), 
\quad f \in C_0^{\infty}(\bbR)  
\end{align}
(the final two assertions can be greatly improved). 

Up to this point $\xi(\, \cdot \,; A,A_0; z_0)$ has been introduced via \eqref{B.31a} and 
hence by \eqref{B.34a}, it is determined only up to an additive integer. It is possible to remove 
this integer ambiguity in $\xi(\, \cdot \,; A,A_0; z_0)$ by adhering to a specific normalization 
as follows: One introduces 
\begin{equation}
U_0(z_0) = (A_0 - z_0 I_{\cH})(A_0 - \ol{z_0} I_{\cH})^{-1}, \quad 
U(z_0) = (A - z_0 I_{\cH})(A - \ol{z_0} I_{\cH})^{-1}, \quad z_0 \in \bbC_+, 
\end{equation}
and then determines a normalized spectral shift function, denoted by 
$\hatt \xi(\, \cdot \,; A,A_0; z_0)$, with the help of fixing the branch of 
$\ln\big(\wti D_{A/A_0} (z_0 ; z_0)\big)$ by the equation,
\begin{align} 
\begin{split} 
i \Im\big(\ln\big(\wti D_{A/A_0} (z_0 ; z_0)\big)\big) &= 
2 i \Im(z_0) \int_{\bbR} \f{\hatt \xi(\lambda; A,A_0; z_0) d \lambda}{|\lambda - z_0|^2}   \\
&= {\tr}_{\cH} \big(\ln\big(U(z_0)U_0(z_0)^{-1}\big)\big).     \lb{B.36a}
\end{split} 
\end{align} 
Here $\ln(W)$, with $W$ unitary in $\cH$, is defined via the spectral theorem,
\begin{align}
\begin{split} 
W = \ointctrclockwise_{S^1} \mu \, dE_W(\mu),  \quad 
\ln(W) = i \arg(W) = i \ointctrclockwise_{S^1} \arg(\mu) \, dE_W(\mu),&   \\
\arg(\mu) \in (- \pi, \pi],& 
\end{split} 
\end{align}
with $E_W(\cdot)$ the spectral family for $W$, and 
\begin{equation}
\ln\big(U(z_0) U_0(z_0)^{-1}\big) 
= \ln\big(I_{\cH} + [U(z_0) - U_0(z_0)] U_0(z_0)^{-1}\big) \in \cB_1(\cH),
\end{equation}
since 
\begin{equation}
U(z_0) - U_0(z_0) = - 2i \Im(z_0) \big[(A - \ol{z_0} I_{\cH})^{-1} 
- (A_0 - \ol{z_0} I_{\cH})^{-1}\big] \in \cB_1(\cH). 
\end{equation}
In conjunction with \eqref{B.36a} we also mention the estimate,
\begin{equation}
\int_{\bbR} \f{|\hatt \xi(\lambda; A,A_0; z_0)\big| d \lambda}{|\lambda - z_0|^2} 
\leq \f{\pi}{2} \big\|(A - z_0 I_{\cH})^{-1} - (A_0 - z_0 I_{\cH})^{-1} \big\|_{\cB_1(\cH)}. 
\end{equation}

Moreover, if there exists $\alpha_0 \geq 0$ such that
\begin{equation}
\alpha \big\|(A - i \alpha I_{\cH})^{-1} - (A_0 - i \alpha I_{\cH})^{-1}\big\|_{\cB(\cH)} < 1, 
\quad \alpha > \alpha_0,     \lb{B.38a}
\end{equation}
then (cf.\ \cite[p.~300--303]{Ya92}), 
\begin{equation}
\hatt \xi(\, \cdot \,; A,A_0; i \alpha) = \hatt \xi (\, \cdot \, ; A,A_0) \, \text{ is independent of $\alpha$ 
 for $\alpha > \alpha_0$.}    \lb{B.39a}
\end{equation}

We also note that if 
\begin{equation}
\dom(A) = \dom(A_0), \quad (A-A_0)(A_0 - z_0 I_{\cH})^{-1} \in \cB_1(\cH)   \lb{B.40A}
\end{equation}
holds for some (and hence for all) $z_0 \in \rho(A_0)$, then \eqref{B.38a} is valid for 
$0 < \alpha_0$ sufficiently large, and (cf.\ \cite[p.~303--304]{Ya92}),  
\begin{equation}
\hatt \xi (\, \cdot \, ; A,A_0) \in L^1\big(\bbR; (|\lambda| + 1)^{-1 - \varepsilon} d \lambda\big), 
\quad \varepsilon > 0.    \lb{B.40a}
\end{equation}
Since different spectral shift functions only differ by a constant, the inclusion \eqref{B.40a} 
remains valid for all spectral shift functions under the assumptions \eqref{B.40A}. 

Finally, if 
\begin{equation} 
B=B^* \in \cB_1(\cH) \, \text{ and } \, A = A_0 + B,    \lb{B.41a} 
\end{equation} 
then \eqref{B.38a} holds with $\alpha_0 = 0$ and 
(cf.\ \cite[p.~303--304]{Ya92})
\begin{equation}
\hatt \xi (\, \cdot \, ; A,A_0) \in L^1\big(\bbR; d \lambda\big).    \lb{B.42a} 
\end{equation}
Assuming \eqref{B.41a}, one usually introduces the (standard) perturbation determinant,
 \begin{align}
 \begin{split} 
 D_{A/A_0} (z) &= {\det}_{\cH} \big((A - z I_{\cH}) (A_0 - z I_{\cH})^{-1}\big)    \\
 &= {\det}_{\cH} \big(I_{\cH} + B(A_0 - z I_{\cH})^{-1}\big), \quad 
 z \in \rho(A) \cap \rho(A_0),
 \end{split} 
 \end{align} 
and the associated spectral shift function
\begin{equation}
\xi(\lambda; A,A_0) = \pi^{-1} \lim_{\varepsilon \downarrow 0} 
\Im(\ln(D_{A/A_0}(\lambda + i \varepsilon))) \, \text{ for a.e.~$\lambda \in \bbR$,} 
\end{equation}
and hence obtains the following well-known facts for $z \in \rho(A) \cap \rho(A_0)$, 
$\Im(z_0) > 0$, 
\begin{align}
& \ol{D_{A/A_0} (z)} = D_{A/A_0} (\ol z),  \quad \lim_{|\Im(z)|\to\infty} D_{A/A_0} (z) =1,   \\
& \xi(\lambda; A,A_0) = (2\pi)^{-1} \lim_{\varepsilon \downarrow 0} \big[
\Im(\ln(D_{A/A_0}(\lambda + i \varepsilon))) 
- \Im(\ln(D_{A/A_0}(\lambda - i \varepsilon)))\big]   \no \\
& \hspace*{8.3cm} \text{ for a.e.~$\lambda \in \bbR$,}     \lb{B.45a} \\ 
& \wti D_{A/A_0}(z;z_0) = D_{A/A_0} (z) / D_{A/A_0} (\ol{z_0}),   \lb{B.46a} \\ 
& \xi(\, \cdot \, ; A,A_0) \in L^1(\bbR; d\lambda),    \\ 
& \ln(D_{A/A_0} (z)) = \int_{\bbR} \xi(\lambda; A, A_0) d \lambda \, (\lambda - z)^{-1},    \\
& \int_{\bbR} \xi(\lambda; A, A_0) d \lambda = {\tr}_{\cH} (B), \quad 
\int_{\bbR} |\xi(\lambda; A, A_0)| d \lambda \leq \|B\|_{\cB_1(\cH)},   
\end{align}
Combining the facts \eqref{B.31a}, \eqref{B.45a}, and \eqref{B.46a} at first only yields 
for some $n(z_0) \in \bbZ$, 
\begin{equation} 
\xi(\lambda; A,A_0; z_0) = \xi(\lambda; A,A_0) + n(z_0) \, \text{ for a.e.~$\lambda \in \bbR$.}
\end{equation}
However, also taking into account \eqref{B.39a} and \eqref{B.42a} finally yields 
\begin{equation} 
\hatt \xi(\lambda; A,A_0) = \xi(\lambda; A,A_0) \, \text{ for a.e.~$\lambda \in \bbR$.}  \lb{B.47a}
\end{equation}
Thus, the normalization employed in \eqref{B.36a} is consistent with the normalization implied by \eqref{B.42a} in the case of trace class perturbations.

We continue this appendix with the following result, originally derived in \cite{GN12a} 
under slightly different hypotheses: 

\begin{theorem} \lb{tB.1}
$(i)$ Suppose $A_0$ and $A$ are self-adjoint operators with 
$\dom(A) = \dom(A_0) \subseteq \cH$, with $B = \ol{(A - A_0)} \in \cB(\cH)$. \\
$(ii)$ Assume that for some $($and hence for all\,$)$ $z_0 \in \rho(A_0)$,
\begin{align}
& |B|^{1/2} (A_0 - z_0 I_{\cH})^{-1} \in \cB_2(\cH),    \lb{B.2} 
\end{align}
and that 
\begin{equation}
\lim_{z \to \pm i \infty} 
\big\||B|^{1/2} (A_0 - z I_{\cH})^{-1} |B|^{1/2}\big\|_{\cB_2(\cH)} = 0. 
\lb{B.4}
\end{equation}
$(iii)$ Suppose that
\begin{equation}
{\tr}_{\cH} \big((A_0 - z I_{\cH})^{-1} B (A_0 - z I_{\cH})^{-1}\big) = 
\eta' (z), \quad z \in \rho(A_0),    \lb{B.5}
\end{equation} 
where $\eta(\cdot)$ has normal limits, 
$\lim_{\varepsilon \downarrow 0} \eta(\lambda + i \varepsilon) := \eta(\lambda + i0)$  
for a.e.\ $\lambda \in \bbR$.

Then
\begin{align}
& \int_{\bbR} \xi(\lambda; A, A_0) d\lambda 
\big[(\lambda - z)^{-1} - (\lambda - z_0)^{-1}\big] = \eta (z) - \eta (z_0)   \no \\
& \quad  
+ \ln\bigg(\f{{\det}_{2,\cH}\big(I_{\cH} + \sgn(B) |B|^{1/2} 
(A_0 - z I_{\cH})^{-1} |B|^{1/2}\big)}
{{\det}_{2,\cH}\big(I_{\cH} + \sgn(B) |B|^{1/2} (A_0 - z_0 I_{\cH})^{-1} |B|^{1/2}\big)}
\bigg), 
\quad z, z_0 \in \rho(A) \cap \rho(A_0),    \lb{B.6} 
\end{align} 
and for some constant $c\in\bbR$, 
\begin{align}
\begin{split} 
\xi(\lambda; A, A_0) &= \pi^{-1} \Im\big(\ln\big(
{\det}_{2,\cH}\big(I_{\cH} + \sgn(B) |B|^{1/2} 
(A_0 - (\lambda + i 0) I_{\cH})^{-1} |B|^{1/2}\big)\big)\big)   \lb{B.7} \\
& \quad + \pi^{-1} \Im(\eta(\lambda + i0)) + c \, \text{ for a.e.\ $\lambda \in \bbR$.} 
\end{split} 
\end{align}
\end{theorem} 

\begin{remark} \lb{rB.2}
$(i)$ We note that Theorem \ref{tB.1} was derived in \cite{GN12a} for unbounded 
quadratic form perturbations $B$ of $A_0$ and hence $A_0$ was assumed to be 
bounded from below. Since we here assume that $B$ is bounded, boundedness 
from below of $A_0$ is no longer needed and the proof of \cite[Theorem~2.3]{GN12a}
applies line by line to the current setting. We also note that since $B|_{\dom(A_0)}$ 
is symmetric and bounded, $B = \ol{(A - A_0)}$ is self-adjoint on $\cH$. The basic 
identity underlining Theorem \ref{tB.1} is, of course, 
\begin{align}
& {\tr}_{\cH}\big((A - z I_{\cH})^{-1} - (A_0 - z I_{\cH})^{-1}\big) + 
{\tr}_{\cH}\big((A_0 - z I_{\cH})^{-1}B(A_0 - z I_{\cH})^{-1}\big)    \lb{B.8} \\
& \quad = - \f{d}{dz} \ln\big({\det}_{2,\cH}\big(I_{\cH} + \sgn(B)|B|^{1/2}
(A_0 - z I_{\cH})^{-1}|B|^{1/2}\big)\big), \quad z \in \rho(A) \cap \rho(A_0). \no 
\end{align}
$(ii)$ One can show (cf.\ \eqref{det}) that 
\begin{equation}
\eta(z) - \eta(z_0) = (z - z_0) {\tr}_{\cH} \big((A_0 -z I_{\cH})^{-1} B (A_0 - z_0 I_{\cH})^{-1}\big), 
 \quad z, z_0 \in \rho(A) \cap \rho(A_0).
\end{equation}
We will use this fact later in this appendix. 
\hfill $\diamond$
\end{remark}

For modified Fredholm determinants and their properties, we refer, for instance, to 
\cite[Sect.~XI.9]{DS88}, \cite[Sect.\ IV.2]{GK69} and \cite[Ch.\ 9]{Si05}. Here we just note that 
\begin{equation}\lb{Z2.117}
{\det}_{2,\cH}(I_{\cK}-A) = \prod_{n\in\cJ}(1-\lambda_n(A))e^{\lambda_n(A)},\quad A\in \cB_2(\cH),
\end{equation}
where $\{\lambda_n(A)\}_{n\in \cJ}$ is an enumeration of the non-zero eigenvalues of $A$, listed in non-increasing order according to their moduli, and $\cJ\subseteq \bbN$ is an appropriate indexing set, and 
\begin{align}
& {\det}_{2,\cH}(I_{\cK}-A)= {\det}_{\cH}((I_{\cH}-A)\exp(A)), \quad A\in\cB_2 (\cH), \lb{Z3.21} \\
& {\det}_{2,\cH}((I_{\cK}-A)(I_{\cH}-B))={\det}_{2,\cH}(I_{\cK}-A){\det}_{2,\cH}(I_{\cH}-B)
e^{-\tr_{\cH}(AB)},\lb{Z3.22}\\
&\hspace*{8.65cm}A, B\in\cB_2(\cH).   \no 
\end{align} 
In addition, we recall the fact, that ${\det}_{2,\cH}(I_{\cH} + \cdot \,)$ 
is continuous on $\cB_2(\cH)$, explicitly, for some $c > 0$, one has the estimate (cf.\ \cite[Theorem~9.2\,(c)]{Si05})
\begin{align}
\begin{split} 
|{\det}_{2,\cH}(I_{\cH} + T) - {\det}_{2,\cH}(I_{\cH} + S)| \leq \|T - S\|_{\cB_2(\cH)} 
e^{c[\|S\|_{\cB_2(\cH)} + \|T\|_{\cB_2(\cH)} + 1]^2},&   \lb{detcont} \\
S, T \in \cB_2(\cH).&
\end{split} 
\end{align} 

In addition, we need some results concerning the connection between trace 
formulas and modified Fredholm determinants (cf., eg., \cite[Sect.\ IV.2]{GK69}, 
\cite[Ch.\ 9]{Si05}, \cite[Sect.\ 1.7]{Ya92}). Suppose that $A_0$ is self-adjoint in 
the complex, separable Hilbert space $\cH$, and assume that the self-adjoint operator 
$B$ in $\cH$, satisfies
\begin{align}
\begin{split} 
& \dom(B) \supseteq \dom(A_0),    \lb{domB} \\
& \text{$B$ is infinitesimally bounded with respect to $A_0$,}   
\end{split}
\end{align} 
and for some (and hence for all) $z_0, z_1 \in \rho(A_0)$, 
\begin{equation}
B (A_0 - z_0 I_{\cH})^{-1} \in \cB_2(\cH), \quad 
(A_0 - z_0 I_{\cH})^{-1} B (A_0 - z_1 I_{\cH})^{-1} \in \cB_1(\cH).      \lb{BB2}
\end{equation} 
(We recall that $B$, with $\dom(B) \supseteq \dom(A_0)$, is called infinitesimally bounded 
with respect to $A_0$, if for all $\varepsilon > 0$, there exists $\eta(\varepsilon) > 0$, such 
that for $f \in \dom(A_0)$, 
$\|B f\|_{\cH} \leq \varepsilon \|A_0 f\|_{\cH} + \eta(\varepsilon) \|f\|_{\cH}$.) 

Then by assumption \eqref{domB}, 
\begin{equation} 
A = A_0 + B, \quad \dom(A) = \dom(A_0),    
\end{equation}
is self-adjoint in $\cH$, and 
\begin{align}
\begin{split} 
& \big[(A - z I_{\cH})^{-1} - (A_0 - z I_{\cH})^{-1}\big] 
= \big[(A_0 - z I_{\cH})^{-1} B (A_0 - z I_{\cH})^{-1}\big]     \\
& \quad \times \big[(A_0 - z I_{\cH}) (A - z I_{\cH})^{-1}\big] \in \cB_1(\cH), \quad 
z \in \rho(A) \cap \rho(A_0).
\end{split} 
\end{align}
Given this setup, one concludes the trace formula (cf., e.g., \cite[p.~44]{Ya92})
\begin{align}
& \tr_{\cH} \big((A - z I_{\cH})^{-1} - (A_0 - z I_{\cH})^{-1}\big)      \no \\
& \quad = - \f{d}{dz} \ln \big({\det}_{2,\cH}\big(I_{\cH} 
+ B (A_0 - z_0 I_{\cH})^{-1}\big)\big)   \no \\
& \qquad - {\tr}_{\cH}\big((A_0 - z I_{\cH})^{-1} B (A_0 - z I_{\cH})^{-1}\big)   \lb{detAA0} \\
& \quad = - \int_{\bbR} \f{\xi(\lambda; A, A_0) d \lambda}{(\lambda - z)^2}, 
\quad z \in \rho(A) \cap \rho(A_0),    \no
\end{align}
and consequently, also
\begin{align}
& \ln \bigg(\f{{\det}_{2,\cH}\big(I_{\cH} + B (A_0 - z I_{\cH})^{-1}\big)} 
{{\det}_{2,\cH}\big(I_{\cH} + B (A_0 - z_0 I_{\cH})^{-1}\big)}\bigg)  
\no \\
& \qquad + (z - z_0) {\tr}_{\cH}\big((A_0 - z I_{\cH})^{-1} B (A_0 - z_0 I_{\cH})^{-1}\big)    
\lb{det} \\ 
& \quad = \int_{\bbR} \xi(\lambda; A, A_0) d \lambda \big[(\lambda - z)^{-1} 
- (\lambda - z_0)^{-1}\big], \quad z,z_0 \in \rho(A) \cap \rho(A_0).     \no 
\end{align}
(To verify \eqref{det} it suffices to differentiate either side of \eqref{det} w.r.t. $z$, comparing 
with the final three lines of relation \eqref{detAA0}, and observing that either side of \eqref{det} vanishes at $z = z_0$.)   

At this point we recall that ${\det}_{2,\cH}(I_{\cH} + \cdot \,)$ is continuous on $\cB_2(\cH)$, 
as recorded earlier in \eqref{detcont}. 

Next, suppose that $A_{0,n}$, $B_n$, $A_n = A_{0,n} + B_n$, $n \in \bbN$, and 
$A_0$, $B$, $A = A_0 + B$ satisfy hypotheses \eqref{domB} and \eqref{BB2}. Moreover, 
assume that for some (and hence for all) $z_0 \in \bbC \backslash \bbR$, 
\begin{align}
& \wlim_{n \to \infty} (A_{0,n} - z_0 I_{\cH})^{-1} = (A_0 - z_0 I_{\cH})^{-1},     \lb{wrl} \\
& \lim_{n\to\infty} \big\|B_n (A_{0,n} - z_0 I_{\cH})^{-1} -
B (A_0 - z_0 I_{\cH})^{-1}\big\|_{\cB_2(\cH)} = 0,     \lb{B2conv} \\ 
& \lim_{n\to\infty} \big\|(A_{0,n} - z_0 I_{\cH})^{-1} B_n (A_{0,n} - z_0 I_{\cH})^{-1}     \no \\
& \hspace*{1.1cm} - (A_0 - z_0 I_{\cH})^{-1}B (A_0 - z_0 I_{\cH})^{-1}\big\|_{\cB_1(\cH)} = 0.    \lb{B1conv} 
\end{align}
One notes that due to self-adjointness of $A_{0,n}, A_0$, $n \in \bbN$, relation \eqref{wrl} 
is actually equivalent to strong resolvent convergence, that is,
\begin{equation}
\slim_{n \to \infty} (A_{0,n} - z I_{\cH})^{-1} = (A_0 - z I_{\cH})^{-1},  
\quad z \in \bbC \backslash \bbR.    \lb{srl}
\end{equation}   
Moreover, the well-known identity (see, e.g., \cite[p.~178]{We80}), 
\begin{align}
& (T_1 - z I_{\cH})^{-1} - (T_2 - z I_{\cH})^{-1} = (T_1 - z_0 I_{\cH})(T_1 - z I_{\cH})^{-1}   \no \\
& \quad \times \big[(T_1 - z_0 I_{\cH})^{-1} - (T_2 - z_0 I_{\cH})^{-1}\big] 
(T_2 - z_0 I_{\cH})(T_2 - z I_{\cH})^{-1},     \lb{resid} \\
& \hspace*{6.6cm} z, z_0 \in \rho(T_1) \cap \rho(T_2),    \no 
\end{align}
where $T_j$, $j=1,2$, are linear operators in $\cH$ with 
$\rho(T_1) \cap \rho(T_2) \neq \emptyset$, together with Lemma \ref{l3.6}, \eqref{B1conv}, 
and \eqref{srl} imply
\begin{align}
& \lim_{n\to\infty} \big\|(A_{0,n} - z_0 I_{\cH})^{-1} B_n (A_{0,n} - z_1 I_{\cH})^{-1}     \no \\
& \hspace*{1.1cm} - (A_0 - z_0 I_{\cH})^{-1}B (A_0 - z_1 I_{\cH})^{-1}\big\|_{\cB_1(\cH)} = 0, 
\quad z_0, z_1 \in \bbC \backslash \bbR.    \lb{B1conve} 
\end{align}

Then \eqref{det} applied to the self-adjoint pairs $(A_n, A_{0,n})$, $n \in \bbN$, and $(A, A_0)$, in combination with \eqref{detcont}--\eqref{B1conve} implies the continuity result,  
\begin{align}
& \lim_{n\to\infty} \int_{\bbR} \xi(\lambda; A_n, A_{0,n}) d \lambda \, \big[(\lambda - z)^{-1} 
- (\lambda - z_0)^{-1}\big]     \no \\
& \quad =\lim_{n\to\infty}\Bigg\{ \ln \Bigg(\f{{\det}_{2,\cH}\big(I_{\cH} 
+ B_n (A_{0,n} - z I_{\cH})^{-1}\big)} 
{{\det}_{2,\cH}\big(I_{\cH} + B_n (A_{0,n} - z_0 I_{\cH})^{-1}\big)}\Bigg)   
\no \\
& \hspace*{1.8cm} + (z - z_0) 
{\tr}_{\cH}\big((A_{0,n} - z I_{\cH})^{-1} B_n (A_{0,n} - z_0 I_{\cH})^{-1}\big) 
\Bigg\}    \no \\
& \quad = \ln \Bigg(\f{{\det}_{2,\cH}\big(I_{\cH} 
+ B (A_0 - z I_{\cH})^{-1}\big)} 
{{\det}_{2,\cH}\big(I_{\cH} + B (A_0 - z_0 I_{\cH})^{-1}\big)}\Bigg)   
\no \\
&  \qquad + (z - z_0) 
{\tr}_{\cH}\big((A_0 - z I_{\cH})^{-1} B (A_0 - z_0 I_{\cH})^{-1}\big) 
\no \\
& \quad = \int_{\bbR} \xi(\lambda; A, A_0) d \lambda \, \big[(\lambda - z)^{-1} 
- (\lambda - z_0)^{-1}\big],  \quad z, z_0 \in \bbC \backslash \bbR.     \lb{detlim} 
\end{align}

We note that these considerations naturally extend to more complex situations where 
$A = A_0 +_q B$, $A_n = A_{0,n} +_q B_n$, $n \in \bbN$, are defined as quadratic form 
sums of $A_0$ and $B$ and $A_{0,n}$ and $B_n$ (without assuming any correlation 
between the domains of $A$, $A_n$ and $A_0$), and the modified Fredholm determinants 
are replaced by symmetrized ones as in Theorem \ref{tB.1}, see, for instance, 
\cite{GN12}, \cite{GN12a}, and \cite{GZ12}. 
 
\smallskip 
 
\noindent 
{\bf Acknowledgments.} We are indebted to Harald Grosse, Jens Kaad, 
Yuri Latushkin, Matthias Lesch  (in particular for sending us \cite{Wo07}), 
Konstantin Makarov, Alexander Sakhnovich, and Yuri Tomilov for helpful 
discussions and correspondence. 

A.C., F.G., and F.S. thank the Erwin Schr\"odinger International 
Institute for Mathematical Physics (ESI), Vienna, Austria, for funding support 
for this collaboration in form of a Research-in Teams project, ``Scattering Theory 
and Non-Commutative Analysis'' for the duration of June 22 to July 22, 2014. 
F.G. and G.L. are indebted to Gerald Teschl and the ESI for a kind invitation to visit 
the University of Vienna, Austria, for a period of two weeks in June/July of 2014. 
A.C., G.L., and F.S.\ gratefully acknowledge financial support from the Australian 
Research Council. A.C.\ also thanks the Alexander von Humboldt 
Stiftung and colleagues at the University of M\"unster.

 
\end{document}